\documentclass[a4paper,12pt]{amsart}

\usepackage{hyperref}

\usepackage{microtype}
\usepackage{amssymb,amsmath,amsfonts,amsthm,bm,dsfont,latexsym,amscd,verbatim,url,nicefrac,stmaryrd,enumerate,colonequals,dsfont,appendix,mathtools,stackrel}
\usepackage[all,cmtip,2cell]{xy}
\UseAllTwocells

\usepackage{times}
\usepackage{graphicx}
\usepackage{fullpage}

\usepackage{tikz}
\usetikzlibrary{matrix,arrows,decorations.pathmorphing}

\usepackage{verbatim}

\usepackage{mathabx}

\usepackage[english]{babel}
\usepackage[utf8]{inputenc}

\usepackage{mathrsfs}
\newtheorem*{thm*}{Theorem}
\newtheorem{thm}{Theorem}[section]
\newtheorem{cor}[thm]{Corollary}
\newtheorem{lemma}[thm]{Lemma}

\newtheorem{prop}[thm]{Proposition}

\newtheorem{defn}[thm]{Definition}

\theoremstyle{remark}

\theoremstyle{definition}

\newtheorem{rmk}[thm]{Remark}

\numberwithin{equation}{thm}

\def\beq{\begin{equation}}
\def\eeq{\end{equation}}

\def\beqn{\begin{equation*}}
\def\eeqn{\end{equation*}}

\def\ben{\begin{enumerate}}
\def\een{\end{enumerate}}

\def\crash#1{}
\def\A{{\mathbb A}}
\def\B{{\mathbb B}}
\def\C{{\mathbb C}}

\def\F{{\mathbb F}}

\def\G{{\mathbb G}}
\def\L{{\mathbb L}}

\def\N{{\mathbb N}}
\def\P{{\mathbb P}}
\def\Q{{\mathbb Q}}
\def\R{{\mathbb R}}
\def\T{{T}}

\def\Z{{\mathbb Z}}

\def\l{\left}
\def\r{\right}

\def\an{{\rm an}}

\def\cf{\emph{cf.}~}

\def\lc{\emph{loc.cit.}~}

\def\cO{{\mathcal O}}

\def\cU{{\mathcal U}}

\def\sS{{\mathscr O}}

\def\sS{{\mathscr S}}

\def\sX{{\mathscr X}}
\def\sY{{\mathscr Y}}
\def\sZ{{\mathscr Z}}

\def\bD{{\mathbf D}}

\def\bT{{\mathbf T}}

\def\fR{{\mathfrak R}}
\def\fX{{\mathfrak X}}

\def\La{\Lambda}

\def\an{{\rm an}}

\def\Sp{{\rm Spa \,}}
\def\Spec{{\rm Spec \,}}
\def\Hom{{\rm Hom \,}}

\def\id{{\rm id\,}}

\def\id{{\rm Id\,}}

\def\hocolim{\mathop{\rm hocolim}}

\def\lt{\langle}
\def\gt{\rangle}

\def\Spa{{\rm Spa\,}}

\def\Ho{{\rm Ho}}

\DeclareMathOperator{\et}{\acute{e}t}

\def\proet{{\rm \text{pro\emph{\'e}}t\,}}
\def\eff{{\rm eff\,}}

\def\comp{{\rm comp\,}}
\def\an{{\rm an\,}}

\def\Cofib{{\rm Cofib\,}}

\def\Tr{{\rm tr\,}}

\def\art{{\rm art\,}}

\def\bDM{\mathbf{DM}}
\def\bDA{\mathbf{DA}}
\def\bRigDA{\mathbf{RigDA}}
\def\bRigDM{\mathbf{RigDM}}
\def\bFormDA{\mathbf{FormDA}}

\def\bSh{\mathbf{Sh}}

\def\bPerfDA{\mathbf{PerfDA}}

\newcommand{\catD}{\mathbf{D}}

\newcommand{\catT}{\mathbf{T}}

\newcommand{\adj}[4]{#1\negmedspace: #2\rightleftarrows #3:\negmedspace #4}

\def\An{{\rm An\,}}
\DeclareMathOperator{\cat}{C}
\DeclareMathOperator{\catM}{M}

\DeclareMathOperator{\Ch}{Ch}
\DeclareMathOperator{\colim}{colim}
\DeclareMathOperator{\fh}{fh}

\DeclareMathOperator{\Nor}{Nor}

\DeclareMathOperator{\Ev}{Ev}
\DeclareMathOperator{\Cor}{Cor}
\DeclareMathOperator{\Frob}{Frob}
\DeclareMathOperator{\Mod}{-Mod}
\DeclareMathOperator{\Perf}{Perf}
\DeclareMathOperator{\Psh}{Psh}
\DeclareMathOperator{\car}{char}
\DeclareMathOperator{\Shv}{Shv}
\DeclareMathOperator{\Sm}{Sm}
\DeclareMathOperator{\Sus}{Sus}
\DeclareMathOperator{\tr}{tr}
\DeclareMathOperator{\Frobet}{Frob\acute{e}t}

\newcommand{\wRigDA}{{\mathbf{sPerfDA}}}

\def\ra{\rightarrow}
\author{Federico Bambozzi, Alberto Vezzani}

\title{Rigidity for rigid analytic motives}

\address{Federico Bambozzi\\Mathematical Institute - University of Oxford, Andrew Wiles Building, Radcliffe Observatory Quarter, Woodstock Road, Oxford OX2 6GG, UK.}
\email{bambozzif@maths.ox.ac.uk}
\address{Alberto Vezzani\\LAGA - Universit\'e Paris 13, Sorbonne Paris Cit\'e, 99 av. Jean-Baptiste Cl\'ement, 93430 Villetaneuse, France.}
\email{vezzani@math.univ-paris13.fr}
\begin{document}
\subjclass{Primary 14C15, 14G22; Secondary 11G25, 19F27}
\keywords{Rigidity Theorem, Rigid Analytic Varieties, Motives, Tilting equivalence}
\thanks{The first author acknowledges the University of Regensburg with the support of the DFG funded CRC 1085 "Higher Invariants. Interactions between Arithmetic Geometry and Global Analysis". The second author was partially supported by the ANR Grant PERCOLATOR: ANR-14-CE25-0002-01 and by the ANR JCJC Grant PERGAMO: ANR-18-CE40-0017.}

\begin{abstract}
	In this paper we prove the Rigidity Theorem for motives of rigid analytic varieties over a non-Archimedean valued field $K$. We prove this theorem both for motives with transfers and without transfers in a relative setting. Applications include the construction of \'etale realization functors, an upgrade of the known comparison between motives with and without transfers and an upgrade of the rigid analytic motivic tilting equivalence, extending them to $\Z[1/p]$-coefficients.
\end{abstract}

\maketitle

\tableofcontents

\section*{Introduction}
The history of the Rigidity Theorem for motives traces back (at least) to the computations of algebraic $K$-theory groups made by Suslin in the   eighties  to prove  the Quillen-Lichtenbaum Conjecture \cite{Sus2,Sus1}. In the subsequent decade, Suslin and Voevodsky realized that those computations were special cases of a general phenomenon common to all homotopy-invariant sheaves with finite coefficients over the big \'etale site over a scheme \cite{VSF}. In modern terms these results are referred to as ``Rigidity Theorems'' and can be rephrased using the language of motives by saying that the category $\bDA_{\et}(S,\La)$ of derived \'etale stable motives without transfers over a scheme $S$ is equivalent to the derived category $\catD(S_{\et},\La)$ of  $\La$-sheaves on the small \'etale site over  $S$, if $\La$ is an $N$-torsion ring with $N \in \N$ invertible in $S$ (under some hypotheses on the cohomological dimension of $S$, see \cite{ayoub-etale}).

%
%The formula above can be formulated concisely by saying that for $\La$ and $S$ as before, the category of \'etale motives is equivalent to the category of Artin motives. %
Other instances of ``Rigidity'' are known. For example, the one for motives with transfers  \cite{CD-etale}, or for $K$-theory and Chow groups (see \cite{bk,gabK}).
This paper treats the Rigidity Theorem in the context of motives of rigid analytic varieties over non-Archimedean valued fields, introduced in \cite{ayoub-rig} by Ayoub, both in their version with and without transfers. More precisely, given any normal rigid analytic variety $S$ over $K$, we denote by $\bRigDA_{\et}(S, \La)$ (resp. $\bRigDM_{\et}(S, \La)$) the category of \'etale motives without transfers (resp. with transfers) over $S$ with coefficients in the ring $\La$. The precise definition of these categories is recalled in the first section of the paper. Our main result  is the following theorem.
 \begin{thm*}[{\ref{thm:rigmain}}]
 	Let $S$ be a normal rigid analytic variety over a non-Archimedean field $K$ with $\ell$-finite cohomological dimension, for all primes $\ell$ invertible in  the residue field $k$ of $K$, and let $\La$ be a $N$-torsion ring, where $N$ is a positive integer invertible in $k$. The functors:
 	$$
 	\L\iota^*\colon {\bD(S_{\et},\Lambda)}\ra{\bRigDA_{\et}(S,\Lambda)}
 	$$
 	$$
 	\L\iota^*\colon {\bD(S_{\et},\Lambda)}\ra{\bRigDM_{\et}(S,\Lambda)}
 	$$
 	are equivalences of monoidal DG-categories.%
 \end{thm*}
As in the algebraic situation, $\bD(S_{\et}, \La)$ denotes the derived category of unbounded complexes of \'etale sheaves of $\La$-modules over the small \'etale site and the functors $\L\iota^*$ arise naturally from the inclusion of the small \'etale topos into the big one. %

We remark that the theorem above is a \emph{generalization} of the usual Rigidity Theorem,  corresponding to the case in which $K$ is trivially valued. Nonetheless, to our knowledge the original algebraic proofs   can \emph{not} be adapted easily to the non-Archimedean context. Our strategy is rather to \emph{use} algebraic Rigidity to deduce the rigid one, by means of the analytification functors and the relation between rigid varieties and formal schemes. We also remark that, even for proving our statement over a field $S=\Spa K$ for motives without transfers, the full relative Rigidity Theorem for schemes is used. Indeed, the six functors formalism plays a crucial role in our proof (see Section \ref{sub:ful}). This is no longer true for motives with transfers, as we show in the appendix, for which a more direct and geometric proof is possible in the absolute case.

Just like its algebraic versions, the theorem above has some interesting immediate consequences, discussed in the last section of the paper. They constitute our main motivation for proving the Rigidity Theorem in the non-Archimedean setting.
\begin{enumerate}
\item We can construct  the $\ell$-adic realization functor for analytic motives, following the approach of Ayoub \cite{ayoub-etale}.% 
\item We can prove an equivalence between rigid analytic motives with and without transfers with coefficients over $\Z[1/p]$ where $p$ is the residue exponential characteristic of $K$.
\item Over a perfectoid field, the motivic tilting equivalence (\cf \cite{vezz-fw}) can be promoted to $\Z[1/p]$-coefficients, and the $\ell$-adic realization functors can be shown to be compatible with it.
\end{enumerate}
We remark that, due to the intricate definition of the rigid motivic tilting equivalence, the proof of the last application is more convoluted that the first two.

The paper is structured as follows. In Section \ref{sec:motives} we recall the basic definitions of the theory of motives of non-Archimedean analytic spaces and we fix the notation used throughout the rest of the paper. %
The core of the paper is Section \ref{sec:thm_no_transfers} where the Rigidity Theorem is stated and proved. The proof is divided in two main steps: the first one, shown in Section \ref{sub:emb} consists in proving that the functors ${\L\iota^*}$ are fully faithful and the second one, shown in Section \ref{sub:ful}, consists in checking that they are essentially surjective. %
 The applications of the main theorem listed above are discussed in Section \ref{sec:applications}. %
Finally, in the Appendix we present a  more ``geometric'' proof of the Rigidity Theorem for $\bRigDM_{\et}(K,\La)$ for which we use as ingredients the main ideas of the proof of \cite[Th\'eor\`eme 2.5.34]{ayoub-rig}, algebraic rigidity over fields,  and  Temkin's results on alterations.
\subsubsection*{Acknowledgments}

The authors would like to thank Joseph Ayoub, Federico Binda, Denis-Charles Cisinski, Florent Martin, Helene Sigloch for important suggestions and discussions about the topics of the paper, and an anonymous referee for his/her valuable comments. %

\section{Non-Archimedean \'etale motives} \label{sec:motives}

The following notation is fixed throughout the  paper. We will use the language of rigid analytic varieties following Tate, see \cite{BGR}. On the other hand, we recall  that the (topos theoretic) underlying topological space of a rigid analytic variety is described by the adic spectrum (see \cite{Huber}) and we will typically use notations and results written in the language of adic spaces, following Huber. In particular, we will use the notation $\Sp$ to denote the analytic spectrum. %

\begin{itemize}
	\item With $K$ we always denote a field which is complete with respect to a (fixed) non-trivial non-Archimedean valuation  $||\cdot||\colon K\ra\R_{\geq0}$ with a finite $\ell$-cohomological dimension, for any prime $\ell$ which is invertible in its residue field. 
	\item A \emph{Tate algebra over $K$ } is a topological ring $A$ obtained as a quotient of the completion $K\langle x_1,\ldots,x_n\rangle $ of the polynomial ring $K[x_1,\ldots,x_n]$ with respect to the Gau\ss{} norm. Its analytic spectrum is denoted by $\Sp A$.
	\item An \emph{affinoid rigid analytic variety} [resp. a \emph{rigid analytic variety}] is a locally ringed space which is [resp. locally] isomorphic to the analytic spectrum $\Sp A$ of a Tate algebra $A$.  %
	\item We denote with $\B_K^1$ the closed unit ball. This is the affinoid space  $\Sp K \lt T \gt$. For any analytic space $S$ we denote by $\B_S^1 = \B_K^1 \times_K S$ the (closed) unit disc over $S$. 
	\item Let $A$ be a Tate algebra. The analytification functor from schemes of finite type over $\Spec A$ to analytic varieties over $\Sp A$ is denoted $(.)^\an$. It is defined by means of the following universal property:
	$$
	\Hom_{\Sp A}(\Sp B,X^{\an})\cong\Hom_{\Spec A}(\Spec B,X)
	$$
	for any Tate algebra $B$ over $A$.
	\item If $\La $ is a (commutative, unital) ring, and $\cat$ is a category [resp. a site] we let $\Psh(\cat,\La)$ [resp. $\Shv(\cat,\La)$] be the category of presheaves [resp. sheaves] on $\cat$ with values in $\La$-modules.
	\item	Let $S$ be a rigid analytic variety over $K$. We denote by $S_{\et}$ the small \'etale site over $S$ and by  $\Sm_S$  the site of smooth analytic varieties over $S$ equipped with the \'etale topology. 
	\item For any scheme or rigid analytic variety $S$, $\bD(S_{\et}, \Lambda)$ denotes the (unbounded)  derived category of the category of sheaves of $\Lambda$-modules over the small \'etale site over $S$. %
\end{itemize}
 
We also recall the definitions of the (unbounded, derived, \'etale) motivic categories of algebraic and rigid analytic varieties as defined by Ayoub \cite{ayoub-rig}. We will make use of the language of model categories (even though everything can be restated in terms of $DG$-categories or $\infty$-categories) for which we refer to \cite{hovey}. Categories of complexes of presheaves $\Ch(\Psh(\cat,\La))$ will be endowed with the projective model structure (see \cite{Dug}).

\begin{defn} \label{defn:effective_etale_motives}Let $S$ be a normal rigid analytic variety over $K$. 
The category $\bRigDA_{\et}^\eff(S, \Lambda)$ (the category of \emph{effective \'etale rigid motives over $S$}) is the homotopy category of the  Bousfield localization of $\Ch(\Psh(\Sm_S, \Lambda))$ over the \'etale weak equivalences (that is, morphisms inducing quasi-isomorphisms of the induced complexes of \'etale sheaves) and over $\B^1$-homotopies (shifts of the maps of representable presheaves induced by the projection morphisms $\B^1_Y\ra Y$ with $Y\in\Sm_S$). If $Y$ is in $\Sm_S$ we denote by $\Lambda_S(Y)$ the object in $\bRigDA_{\et}^{\eff}(S, \Lambda)$ associated to the presheaf represented by $Y$. 
\end{defn}

We refer to  Section 2 of \cite{vezz-berk} for more details about   $\bRigDA_{\et}^\eff(S, \Lambda)$ (in its Nisnevich form). We only point out that this DG-category can be defined as a Verdier quotient of $\bD(\Psh(\Sm_S, \Lambda))$ with respect to \'etale descent and $\B^1$-invariance. It therefore enjoys the following universal property: any functor $F$ from $\Sm_S$ to a $\La$-enriched model category $\catM$ satisfying \'etale descent and $\B^1$-invariance admits an extension to a Quillen adjunction $\L F^*\colon \bRigDA_{\et}^\eff(S, \Lambda)\rightleftarrows\Ho(\catM)\colon\R F_*$. See \cite{CG,Dug} for details.

\begin{rmk}
The category $\bRigDA_{\et}^\eff(S, \Lambda)$ is endowed with a monoidal structure, for which  $\Lambda_S(Y)\otimes\Lambda_S(Y')\cong\Lambda_S(Y\times_SY')$ for any $Y,Y'$ in $\Sm_S$ (see \cite[Propositions 4.2.76 and 4.4.63]{ayoub-th2}). The same is true for the category $\bD(S_{\et}, \Lambda)$.
\end{rmk}

\begin{defn} \label{defn:Tate_motive} \label{defn:etale_motives}
We denote with $\T = \T_S \in \bD(\Psh(\Sm_S, \Lambda))$ the complex of sheaves $$\Cofib( \Lambda[\Hom(-, \G_{m, S}^\an)] \to \Lambda[\Hom(-, \A_S^\an)]).$$ 
The object $\Lambda_S(1)\colonequals\T_S[-2] \in \bRigDA_{\et}^\eff(S, \Lambda)$ is called the \emph{Tate (twisting) motive}. The category $\bRigDA_{\et}(S, \Lambda)$ (the category of \emph{\'etale rigid motives without transfers} over $S$) is defined to be the monoidal  $\T_S$-stabilization of $\bRigDA_{\et}^\eff(S, \Lambda)$. Also in this category,  we denote by $\Lambda_S(Y)$ the motive  associated to the presheaf represented by $Y$. The endofunctor $M\mapsto M\otimes T^{\otimes n}$ will be written as $M\mapsto M(n)$ and its quasi-inverse by $M\mapsto M(-n)$ for any $n\in\N$. The motives $M(n)$ are called the \emph{twists} of $M$.
\end{defn}

For the general theory of stabilization of categories the reader is referred to \cite{hovey-sp} and \cite{ayoub-th2}. We only remark that $\bRigDA_{\et}(S, \Lambda)$ is once again the homotopy category of the model category of symmetric $T_S$-spectra of (the $(\B^1,\et)$-localization of) $\Ch(\Psh(\Sm_S, \Lambda))$ and that 
there is a natural (left Quillen) functor
$$
\bRigDA_{\et}^\eff(S, \Lambda) \ra \bRigDA_{\et}(S, \Lambda)
$$ 
which is a monoidal functor also enjoying a universal property, with respect to making the endo-functor $(-) \otimes \T_S$ invertible. 
 We will typically use spectra (rather than symmetric spectra)  as a model of $\bRigDA_{\et}(S, \Lambda)$ in proofs, which is allowed by means of \cite[Th\'eor\`eme 4.3.79]{ayoub-th2}.

\begin{rmk}
In the case when $K$ has is equipped with the trivial valuation, the theory of rigid analytic varieties collapses to the usual theory of algebraic varieties. The motivic categories defined above coincide then with the classical categories of \'etale motives, denoted by $\bDA^{\eff}_{\et}(S,\La)$ and $\bDA_{\et}(S,\La)$ (see \cite{ayoub-etale}).%
\end{rmk}

In this work we also deal with motives with transfers, whose definition we now recall (see also \cite{ayoub-rig} and \cite{vezz-DADM}). 

\begin{defn} \label{defn:rigid_correspondence}\label{defn:effective_etale_motives_transfer}
	Let $S$ be a normal variety over $K$. We let $\Nor_S$ be the category of quasi-compact normal varieties over $S$ and we let the $\fh$-topology be the one generated by those covering families $\{f_i\colon X_i\ra X\}_{i\in I}$ such that $I$ is finite, and the induced map $\bigsqcup f_i\colon \underset{i\in I}\bigsqcup X_i\ra X$ is finite and surjective. We define the category $\Cor_S$ as the category whose objects are those of $\Sm_S$ and whose morphisms $\Hom(X,Y)$ are computed in $\Shv_{\fh}(\Nor_S)$. We let $\bRigDM^{\eff}_{\et}(S,\La)$ (the category of \emph{effective \'etale rigid motives with transfers} over $S$) be the homotopy category of the Bousfield localization of $\Ch\Psh(\Cor_S,\La)$  over those morphisms $f$ which are \'etale weak equivalences as morphisms in $\Ch\Psh(\Sm_S,\La)$, and over $\B^1$-homotopies (shifts of the maps of representable presheaves induced by the projection morphisms $\B^1_Y\ra Y$ with $Y\in\Sm_S$).  If $Y$ is in $\Sm_S$ we denote by $\Lambda_S^{\tr}(Y)$ the object in $\bRigDM_{\et}^{\eff}(S, \Lambda)$ associated to the presheaf represented by $Y$.
\end{defn}

\begin{defn} \label{defn:etale_motives_transfer}We consider the object $\T^\Tr  \in \bD(\bSh(\Cor_S, \Lambda))$  given by the complex $$\Cofib( \Hom_{\Cor_S}(-, \G_{m, S}^\an) \otimes \La \to \Hom_{\Cor_S}(-, \A_S^\an)\otimes \La).$$ 
The category $\bRigDM_{\et}(S, \Lambda)$ (briefly, \emph{\'etale rigid motives with transfers over $S$}) is defined to be the $\T^\Tr_S$-stabilization of $\bRigDM_{\et}^{\eff}(S, \Lambda)$. If $Y$ is in $\Sm_S$ we denote by $\Lambda^{\tr}_S(Y)$ the motive in $\bRigDM_{\et}(S, \Lambda)$ associated to the presheaf represented by $Y$. The endofunctor $M\mapsto M\otimes (T^{\tr})^{\otimes n}$ will be written as $M\mapsto M(n)$ and its quasi-inverse as $M\mapsto M(-n)$ for any $n\in\N$. The motives $M(n)$ are called the \emph{twists} of $M$.
\end{defn}
\begin{rmk}
Once again, whenever $K$ is endowed with the trivial valuation, the definitions above recover the usual categories of (derived, \'etale) motives with transfers over $S$, denoted by $\bDM^{\eff}_{\et}(S,\La)$ and $\bDM_{\et}(S,\La)$  (see  \cite{CD-etale}).
\end{rmk}
\begin{rmk}
There is a more down-to-earth description of the category $\Cor_S$ in terms of multi-valued functions (or rather, Zariski closed subvarieties of the product which are finite over the first component) see \cite[Remarque 2.2.21]{ayoub-rig}.
\end{rmk}

We recall that in a triangulated category, an object $S$ is \emph{compact} if the functor $\Hom(S,-)$ commutes with direct colimits. A triangulated category is  \emph{compactly generated} if there is a set of compact objects $S$ for which the smallest triangulated subcategory with small sums containing them is the whole category. This property is technically very convenient, and the following fact will be used several times throughout the paper.

\begin{prop}\label{prop:cptgen}
	Let $S$ be a normal rigid analytic variety over $K$. The motivic categories $\bRigDA^{\eff}_{\et}(S,\La)$ [resp. $\bRigDM_{\et}^{\eff}(S,\La)$] are compactly generated by shifts of motives of the form $\La_S(A)$ [resp. $\La^{\tr}_S(A)$] with $A$ smooth affinoid over $S$. Similarly, the motivic categories $\bRigDA_{\et}(S,\La)$ [resp. $\bRigDM_{\et}(S,\La)$] are compactly generated by shifts and twists of motives of the form $\La_S(A)$ [resp. $\La^{\tr}_S(A)$] with $A$ smooth affinoid over $S$.
\end{prop}
\begin{proof}
	It suffices to adapt \cite[Proposition 3.19]{ayoub-etale} to the rigid setting. We remark that a bound on the cohomological dimension of affinoid rigid analytic varieties can be found in 
	\cite[Proposition 0.5.7]{Huber}.
\end{proof}

We now summarize the basic  functors between the various motivic categories introduced so far.

\begin{prop} \label{prop:recall}
Let $S$ be a normal rigid analytic variety and $\Lambda$ be a ring.
\begin{enumerate}
\item The map of sites $S_{\et}\ra\Sm_S$ induces a monoidal Quillen adjunction $$
\adj{\L\iota^*}{\bD(S_{\et},\Lambda)}{\bRigDA^{\eff}_{\et}(S,\Lambda)}{\R\iota_*}.
$$
\item The map of sites $\Sm_S \ra \Cor_S$ induces a monoidal Quillen adjunction $$
\adj{\L a_{\Tr}}{\bRigDA^{\eff}_{\et}(S,\Lambda)}{\bRigDM^{\eff}_{\et}(S,\Lambda)}{\R o_{\Tr}}.
$$
\item The $\T$-stabilization and $\T^\Tr$-stabilization functors induce monoidal Quillen adjunctions
$$
\adj{\L \Sus}{\bRigDA^{\eff}_{\et}(S,\Lambda)}{\bRigDA_{\et}(S,\Lambda)}{\R \Ev}
$$
$$
\adj{\L \Sus}{\bRigDM^{\eff}_{\et}(S,\Lambda)}{\bRigDM_{\et}(S,\Lambda)}{\R\Ev}.
$$
\item Let $S$ be an affinoid rigid analytic variety $S=\Sp A$. The analytification functor induces monoidal  Quillen adjunctions$$
\adj{\L \An^*}{\bDA^{\eff}_{\et}(\Spec A,\Lambda)}{\bRigDA^{\eff}_{\et}(S,\Lambda)}{\R\An_*}
$$
$$
\adj{\L \An^*}{\bDM^{\eff}_{\et}(\Spec A,\Lambda)}{\bRigDM^{\eff}_{\et}(S,\Lambda)}{\R\An_*}.
$$
\item\label{Tcomp} There is a canonical isomorphism $$\L a_\Tr \T_S \cong \T^\Tr_S$$
in $\bRigDM^{\eff}_{\et}(S,\Lambda)$ and in case $S=\Spa A$ is affinoid, there are  canonical isomorphisms $$\L \An^*(\T_{\Spec A}) \cong \T_S
\qquad\L \An^*(\T^{\Tr}_{\Spec A}) \cong \T_S^{\Tr}$$
in $\bRigDA^{\eff}_{\et}(S,\Lambda)$ and $\bRigDM^{\eff}_{\et}(S,\Lambda)$ respectively.
\end{enumerate}
\end{prop}
\begin{proof}For the proof of (2) one can easily adapt the argument of \cite[Paragraph 2.1.7]{CD-etale} to the rigid setting. For the other points, the existence of the Quillen pairs follows at once from the universal property of the categories of motives, stabilizations and localizations (to see that the analytification functor preserves $\A^1$-homotopies we refer to \cite[Proposition 1.3.6 and Theor\'eme 2.5.24]{ayoub-rig}). For the last isomorphisms, see  \cite[Lemme 2.5.18]{ayoub-rig} and \cite[Proposition 1.4.17]{ayoub-rig}.
\end{proof}

We now introduce the canonical functors that we will be mostly interested in.

\begin{defn} \label{rmk:functors_LIDA_LIDM}
By composing the functors of the previous proposition, we obtain the following monoidal Quillen adjunctions (all denoted by $\L\iota^*$ and $\R\iota_*$ by abuse of notation)
	$$
	\adj{\L\iota^*}{\bD(S_{\et},\Lambda)}{\bRigDA^{\eff}_{\et}(S,\Lambda)}{\R\iota_{*}}
	$$
		$$
	\adj{\L\iota^*}{\bD(S_{\et},\Lambda)}{\bRigDA_{\et}(S,\Lambda)}{\R\iota_{*}}
	$$
		$$
	\adj{\L\iota^*}{\bD(S_{\et},\Lambda)}{\bRigDM^{\eff}_{\et}(S,\Lambda)}{\R\iota_{*}}
	$$
	and
	$$
	\adj{\L\iota^*}{\bD(S_{\et},\Lambda)}{\bRigDM_{\et}(S,\Lambda)}{\R\iota_{*}}
	$$
	The full triangulated subcategory with small sums of $\bRigDA_{\et}(S,\Lambda)$ [resp. of $\bRigDM_{\et}(S,\Lambda)$] generated by the essential image of $\L\iota^*$ will be called the category of \emph{rigid analytic Artin motives over $S$ [{with transfers}]}.
\end{defn}

	Take $\La$ to be $N$-torsion with $N$ invertible in $K$. 
The classic Rigidity Theorem can be restated by saying that algebraic Artin motives (i.e. Artin motives with respect to a trivial valuation of $K$) are equivalent to $\catD(S_{\et},\La)$ as well as to the categories of stable motives $\bDA_{\et}(S,\La)$ and $\bDM_{\et}(S,\La)$, under mild hypotheses on $S$  (see \cite{ayoub-etale}).

Also in this paper, we will focus on the case of $N$-torsion coefficients, with $N$ coprime to the residual (exponential) characteristic. Also in this setting, algebraic Artin motives are easily seen to embed in the \emph{effective} categories of motives.

\begin{prop}\label{prop:emb_eff}
	If $\Lambda$ is a $N$-torsion ring with $N$ coprime to the residual (exponential) characteristic of $K$, then the functors
		$$
	\adj{\L\iota^*}{\bD(S_{\et},\Lambda)}{\bRigDA^{\eff}_{\et}(S,\Lambda)}{\R\iota_{*}}
	$$
	and
	$$
	\adj{\L\iota^*}{\bD(S_{\et},\Lambda)}{\bRigDM^{\eff}_{\et}(S,\Lambda)}{\R\iota_{*}}
	$$
	are fully faithful.
\end{prop}
\begin{proof}
We remark that $\bD(S_{\et},\Lambda)$ can be  seen as a full subcategory of $\bD(\Shv(\Sm_S),\Lambda)$ and $\bRigDA^{\eff}_{\et}(S,\Lambda)$ can be described as the full subcategory of $\bD(\Shv(\Sm_S),\Lambda)$ of $\B^1_S$-homotopy invariant objects. In order to prove the first claim, it then suffices  to check that the objects of $\bD(S_{\et},\Lambda)$ are $\B^1_S$-homotopy invariant. In order to do so, one can adapt the proof of \cite[Sous-lemme 4.7]{ayoub-etale} using the analytic version of the acyclicity theorem \cite[Example 0.1.1(ii)]{Huber}. For the case with transfers, one can similarly adapt  the proof of \cite[Theorem 3.1.7]{CD-etale}.
\end{proof}

\section{The Rigidity Theorem} \label{sec:thm_no_transfers}

From now on, we fix a normal rigid analytic variety $S$ over $K$. Once again, we recall that $K$ is assumed to have  finite $\ell$-cohomological dimension, for any prime $\ell$ which is invertible in its residue field. 
The main aim of this section is the proof of the Rigidity Theorem, for motives with and without transfers. We state it here.

\begin{thm}\label{thm:rigmain}
	Let $S$ be a normal rigid analytic variety over $K$ and let $\La$ be a $N$-torsion ring, where $N$ is invertible in the residue field of $K$. The functors:
$$
\L\iota^*\colon {\bD(S_{\et},\Lambda)}\ra{\bRigDA_{\et}(S,\Lambda)}
$$
$$
\L\iota^*\colon {\bD(S_{\et},\Lambda)}\ra{\bRigDM_{\et}(S,\Lambda)}
$$
are equivalences of monoidal DG-categories. Moreover the canonical functor between ${\bRigDM_{\et}^{\eff}(S,\Lambda)} $ and ${\bRigDM_{\et}(S,\Lambda)} $ is an equivalence.
\end{thm}

 We will divide the proof in two steps: we first show that the two functors are fully faithful, and we then prove that they are essentially surjective.

\subsection{The Embedding Theorems}\label{sub:emb}
We now show that the functors above are fully faithful. As in the algebraic proofs (see \cite{ayoub-etale} and \cite{CD-etale}) one of the key points is the  equivalence between the Tate twisting motive and the sheaf of $N$-th roots of unity.

\begin{prop} \label{prop:iso_mu}Let $\La$ be a $N$-torsion ring, where $N$ is invertible in the residue field of $K$. 
There is a natural morphism 
\[ % 
\Lambda_S(1) \cong \T_S[-2] {\to} \mu_{S, N}% 
\] 
in $\bRigDA^{\eff}_{\et}(S,\Lambda)$ where the motive on the right is the one induced by the locally constant sheaf of $N$-th roots of unity. It becomes invertible in $\bRigDA_{\et}(S,\Lambda)$ after applying $\L\Sus$.%
\end{prop}
\begin{proof}
	We first produce the natural morphism between the two motives. The presheaf of abelian groups $\cO^\times$ is represented by $\G_{m, S}^{\an}$.  There is an obvious map of presheaves
	$$\Lambda_S(\G_{m,S}^{\an})=\Hom(-,\G_{m,S}^{\an})\otimes\Lambda\ra\mathcal{O}^\times\otimes_{\Z}\Lambda$$ %
	 where the first tensor is the free $\Lambda$-module over the set, while the second is a base change over $\Z\ra\Lambda$. The induced map on the associated sheaves factors  over $T_S[-1]$.
	  On the other hand, we remark that by the Kummer exact sequence, we can write $\mu_{S, N}[1]\cong\mathcal{O}^\times\otimes_\Z\Lambda$ in $\bRigDA_{\et}^{\eff}(S,\La)$. By shifting on both sides, we therefore obtain a morphism $\T_S[-2] {\to} \mu_{S, N}$ in  $\bRigDA_{\et}^{\eff}(S,\Lambda)$  as wanted.

We now pass to the stable categories and we prove that this morphism becomes invertible. 
Using  (the obvious rigid analogue of) \cite[Lemma 3.4]{ayoub-etale} it suffices to prove the statement in the case when $S$ is affinoid, equal to $\Sp A$. %
We remark that Tate algebras are excellent (see \cite[Remarks 3.5.2]{FvDP}) and therefore (see the paragraph after \cite[Theorem 4.1]{ayoub-etale})  
one can apply 
 \cite[Proposition 4.10]{ayoub-etale} and claim that there is an isomorphism
\begin{equation} \label{eq:iso_mu}
\mu_{S, N} \cong \Lambda_{\Spec A}(1)
\end{equation}
in $\bDA_{\et}(\Spec A, \La)$ induced by the analogous natural morphism between the two motives. Applying the analytification functor $\L \An^*: \bDA_{\et}(\Spec A, \Lambda) \to \bRigDA_{\et}(S, \Lambda)$ to the isomorphism \eqref{eq:iso_mu} we obtain the statement (using Proposition  \ref{prop:recall}, \eqref{Tcomp}).
\end{proof}

There is a stronger statement for motives with transfers.

\begin{prop} \label{prop:iso_muDM}Let $\La$ be a $N$-torsion ring, where $N$ is invertible in the residue field of $K$. 
	There is a natural isomorphism 
	\[ % 
	\Lambda_S^{\tr}(1) \cong \T_S^{\tr}[-2] {\to} \mu_{S, N}% 
	\] 
	in $\bRigDM^{\eff}_{\et}(S,\Lambda)$ where the motive on the right is the one induced by the locally constant sheaf of $N$-th roots of unity. %
\end{prop}
\begin{proof}
	The natural morphism between the two motives is deduced by the morphism of Proposition \ref{prop:iso_mu}.  In order to prove that it is invertible, 
	using  (the obvious rigid analogue of) \cite[Proposition 3.2.8]{CD-etale}, one can assume $S$ is affinoid, equal to $\Sp A$. %
	We remark that Tate algebras are Noetherian (see \cite[Remarks 3.5.2]{FvDP}).  
	One can then apply 
	\cite[Proposition 3.2.3]{CD-etale} and obtain  an isomorphism
	\begin{equation} \label{eq:iso_mu2}
	\mu_{S, N} \cong \Lambda^{\tr}_{\Spec A}(1)
	\end{equation}
	in $\bDM^{\eff}_{\et}(\Spec A, \La)$  induced by the analogous natural morphism. Applying the analytification functor $\L \An^*: \bDM^{\eff}_{\et}(\Spec A, \Lambda) \to \bRigDM^{\eff}_{\et}(S, \Lambda)$ to the isomorphism \eqref{eq:iso_mu2} we obtain the statement (using Proposition \ref{prop:recall}, \eqref{Tcomp}).
\end{proof}

From now on, we will adopt the usual notation for Tate twists in $\bD(S_{\et}, \Lambda)$, hence, for any object $K$ in it, we will write $K(n)$ for the object $K\otimes\mu_N^{\otimes n}$. The previous proposition can be rephrased by saying that $\L\iota^*\colon\bD(S_{\et}, \Lambda)\ra\bRigDA_{\et}(S,\La)$  and $\L\iota^*\colon\bD(S_{\et}, \Lambda)\ra\bRigDM^{\eff}_{\et}(S,\La)$ preserve the twists. 

For the case of motives with transfers, the previous proposition easily shows that the functor $\L\iota^*$ is fully faithful, and proves the last claim of Theorem \ref{thm:rigmain}.

\begin{thm}[Embedding Theorem for $\bRigDM$]\label{thm:embDM}Let $\La$ be a $N$-torsion ring, where $N$ is invertible in the residue field of $K$. 
The categories $\bRigDM_{\et}^{\eff}(S,\La)$ and $\bRigDM_{\et}(S,\La)$ are canonically equivalent. In particular, the functor$$
\L\iota^*\colon {\bD(S_{\et},\Lambda)}\ra{\bRigDM_{\et}(S,\Lambda)}
$$
is fully faithful.
\end{thm}
\begin{proof}
The sheaf $\mu_{S,N}$ is  invertible in ${\bD(S_{\et},\Lambda)}$ (\'etale locally, it is isomorphic to $\La$) and hence (as the functor $\L\iota^*$ is monoidal) the functor $(-) \otimes\mu_{S,N}$ is already invertible in  $\bRigDM_{\et}^{\eff}(S,\La)$. We conclude the first claim using Proposition \ref{prop:iso_muDM}. The second claim follows at once from Proposition \ref{prop:emb_eff}.
\end{proof}

The case of motives without transfers is more complicated. On the other hand,
the proof of the next theorem is the straightforward rigid analytic analogue of \cite[Corollaire 4.11]{ayoub-etale}. We reproduce here its proof for the convenience of the reader.% 

\begin{thm}[Embedding Theorem for $\bRigDA$]\label{thm:embedding_DA}Let $\La$ be a $N$-torsion ring, where $N$ is invertible in the residue field of $K$. 
The functor 
\[
\L\iota_{}^*: \bD(S_{\et}, \Lambda) \to \bRigDA_{\et}(S, \Lambda)
\]
is fully faithful.
\end{thm}
\begin{proof}%
We divide the proof in several steps.

{\it Step 1}: 
The functor $\L\iota_{}^*$ is
obtained as the composition of the functor
\begin{equation} \label{i1}
\L\iota^*\colon{\bD(S_{\et},\Lambda)}\ra{\bRigDA^{\eff}_{\et}(S,\Lambda)}
\end{equation} 
and the left adjoint of the adjunction 
$$
\adj{\L \Sus}{\bRigDA^{\eff}_{\et}(S,\Lambda)}{\bRigDA_{\et}(S,\Lambda)}{\R\Ev}.
$$
As the functor \eqref{i1} is fully faithful (\cf Proposition \ref{prop:recall}), 
it is enough to show that $\R \Ev \circ \L \Sus \circ \L \iota^* \cong \L \iota^*$. Since $\bRigDA^{\eff}_{\et}(S,\Lambda)$ is compactly generated, it is enough to check that for each compact object $M \in \bRigDA^{\eff}_{\et}(S,\Lambda)$ one has that
\[ \Hom (M, \L \iota^* K) \cong \Hom(M, \R \Ev \circ \L \Sus \circ \L \iota^*(K))  \]
for all $K \in \bD(S_{\et},\Lambda)$. By \cite[Th\'eoreme 4.3.61]{ayoub-th2} we know that
\[\begin{aligned}
 \Hom(M, \R \Ev \circ \L \Sus \circ \L \iota^*(K))& \cong 
 \Hom( \L \Sus M, \L \Sus \circ \L \iota^*(K))\\& \cong \underset{n \in \N}\colim\: \Hom(M \otimes \T^{\otimes n}, \L \iota^*(K)\otimes \T^{\otimes n})\\& \cong \underset{n \in \N}\colim\: \Hom(M , \R \Hom (\T^{\otimes n}, \L \iota^*(K) \otimes \T^{\otimes n})).
 \end{aligned} \]

As $M$ is compact, it is enough to show that the map
\[ \L \iota^* K \ra \hocolim_{n \in \N} \R \Hom (\T^{\otimes n}, \L \iota^*(K) \otimes \T^{\otimes n}) \]
is invertible. We make a variable change and we swap $\R\Hom$ with homotopy colimits (see Lemma \ref{lemma:hom_limind}) and we rewrite the second term above as:
\begin{equation}\label{st1}
\begin{aligned}
\hocolim_{n \in \N} &\R \Hom (\T^{\otimes n}, \L \iota^*(K) \otimes \T^{\otimes n}) \\&\cong \hocolim_{n,m \in \N} \R \Hom (\T^{\otimes n+m}, \L \iota^*(K) \otimes \T^{\otimes n+m}) \\& \cong \hocolim_{n \in \N}\R \Hom(\T^{\otimes n}, \hocolim_{m \in \N}\R\Hom(T^{\otimes m},\L \iota^*(K)\otimes \T^{\otimes n+m}).
\end{aligned} 
\end{equation}

{\it Step 2}: We now prove that  $$\hocolim_{m \in \N}\R\Hom(T^{\otimes m},\L \iota^*(K)\otimes \T^{\otimes n+m})$$ appearing in \eqref{st1} is isomorphic to   $$\hocolim_{m \in \N}\R\Hom(T^{\otimes m},\L \iota^*(K(n)[2n]))\otimes \T^{\otimes m}).$$
By \cite[Th\'eoreme 4.3.61]{ayoub-th2} we know that we can compute morphisms between two spectra $(S_n)$  and $(Y_n)$ in the (derived) category of $\N$-sequences by taking level-wise the $\Hom$-groups $$\Hom(S_n,\hocolim_{k \in \N}\R\Hom(T^{\otimes k},Y_{k+n})).$$  In particular, an equivalence of spectra $(Y_n)$ and $(Y_n')$ induces  equivalences in the effective category 
\begin{equation}
\hocolim_{k \in \N}\R\Hom(T^{\otimes k},Y_{k+n})\cong\hocolim_{k \in \N}\R\Hom(T^{\otimes k},Y'_{k+n}).\label{speff}\end{equation}
 We deduce that the isomorphism of the claim follows from the equivalence of spectra $$\L \iota^*(K)\otimes \T^{\otimes n}\cong\L \iota^*(K(n)[2n]))$$
which is a consequence of Proposition \ref{prop:iso_mu}.

{\it Step 3:} We now prove that for any $m$ the object
$$
\hocolim_{n \in \N}\R\Hom(T^{\otimes n},\L \iota^*(K(n))[2n]\otimes T^{\otimes m})
$$
is canonically equivalent (in the effective category) to:
$$
\hocolim_{n \in \N}\R\Hom(T^{\otimes n},\L \iota^*(K(n+m))[2n+2m]).
$$Using again the line of reasoning of the previous step, and more specifically the equation \eqref{speff} with $k=0$, we see that it suffices to give to $L\colonequals(\L \iota^*(K(n))[2n])_{n\in\N}$ a structure of a spectrum, and prove the existence of a compatible equivalence 
$$L\otimes T^{\otimes m} \cong L\otimes\iota^*(K(m)[2m])$$
for any $m$. This second fact follows from Proposition \ref{prop:iso_mu} while for the first fact one can use the transition maps induced by the following equivalence in $\bRigDA_{\et}(S,\La)$ (Poincar\'e duality in \'etale cohomology, see \cite[Section 7.5]{Huber}):
$$
\L \iota^*(K(n))[2n]\cong\R\Hom(T,\L \iota^*(K(n+1))[2n+2])
$$
and eventually replacing $\L \iota^*(K(n))[2n]$ with a cofibrant-fibrant replacement in order to have genuine maps between complexes of presheaves.

{\it Step 4}: We now finish the proof. We take again \eqref{st1} and we replace using Step 2:
$$
\begin{aligned}
 \hocolim_{n \in \N} &\R\Hom(\T^{\otimes n}, \hocolim_{m \in \N}\R\Hom(T^{\otimes m},\L \iota^*(K)\otimes \T^{\otimes n+m})\\ \cong 
 & \hocolim_{n \in \N}\R \Hom(\T^{\otimes n}, \hocolim_{m \in \N}\R\Hom(T^{\otimes m},\L \iota^*(K(n)[2n]))\otimes \T^{\otimes m})).
 \end{aligned}
 $$
We now swap $\R\Hom$ with $\hocolim$ (using again Lemma \ref{lemma:hom_limind}) and Step 3 to obtain
$$
\begin{aligned}
 \hocolim_{n \in \N}& \R\Hom(\T^{\otimes n}, \hocolim_{m \in \N}\R\Hom(T^{\otimes m},\L \iota^*(K(n)[2n]))\otimes \T^{\otimes m}))
 \\
 \cong
& \hocolim_{m \in \N} \R\Hom(\T^{\otimes m}, \hocolim_{n \in \N}\R\Hom(T^{\otimes n},\L \iota^*(K(n)[2n]))\otimes \T^{\otimes m}))\\
 \cong
&  \hocolim_{m \in \N}\R \Hom(\T^{\otimes m}, \hocolim_{n \in \N}\R\Hom(T^{\otimes n},\L \iota^*(K(n+m))[2n+2m]))\\
  \cong
&    \hocolim_{n \in \N}\R \Hom(\T^{\otimes n}, \L \iota^*(K(n))[2n])
\end{aligned}
$$
which is in turn canonically isomorphic to $\L \iota^*K$ (by Poincar\'e duality).
\end{proof}

The following technical lemma was used in the previous proof.

\begin{lemma} \label{lemma:hom_limind}
	The functor
	\[ \R \Hom_{\bRigDA^{\eff}_{\et}(S,\La)} (\T_S^{\otimes n}, -) \]
	commutes with homotopy colimits. 
\end{lemma}
\begin{proof}
	As $\bRigDA^{\eff}_{\et}(S,\La)$ is compactly generated (see Proposition \ref{prop:cptgen})	it suffices to test that for any compact object $M$ attached to an affinoid variety $X=\Spa A$ smooth over $S$, the natural map $$
	\xymatrix@R=1em{
	\Hom(M\otimes T^{\otimes n},\hocolim U_i)\ar[r]&\Hom(M,\hocolim \R\Hom(T^{\otimes n},U_i))\ar[d]^{\sim}\\&\colim\Hom(M\otimes T^{\otimes n},U_i)
}$$ is invertible. It suffices then to prove that $X\otimes T$ is also compact. In order to do this, we can prove (see \cite[Page 54]{ayoub-rig}) that the motive attached to $\Spa A\langle T^{\pm}\rangle$ is compact, which follows again from Proposition \ref{prop:cptgen}.
\end{proof}

We pin down the following consequence of Theorems \ref{thm:embDM} and \ref{thm:embedding_DA}.

\begin{cor}Under the hypotheses of Theorem \ref{thm:rigmain}, 
the subcategory of rigid analytic Artin motives in  $\bRigDA_{\et}(S,\La)$ [resp. $\bRigDM_{\et}(S,\La)$] (see Definition \ref{rmk:functors_LIDA_LIDM}) is equivalent to the category $\bD(S_{\et},\La)$.\qed
\end{cor}

By abuse of notation, under the hypotheses of the previous Corollary, we will refer to rigid analytic Artin motives as well as to the objects of $\bD(S_{\et},\La)$ simply as \emph{Artin motives}.

\subsection{The proof of Rigidity}\label{sub:ful}

We can finally achieve the proof of Theorem \ref{thm:rigmain}. 
To this aim, we will use a (limited) version of the six functor formalism for rigid analytic motives, which is already present in \cite{ayoub-rig}. We recall here the main properties that we will use.

\begin{rmk}
For the sake of readability, we will drop $\L$ and $\R$ for derived functors, whenever the context allows it. In particular, we will write $\iota^*$ for the functors introduced in Definition \ref{rmk:functors_LIDA_LIDM}.
\end{rmk}

\begin{thm} \label{thm:6f1}
	Fix an affinoid algebra $A$. Let $S$ be $\Spa A$. 
	\begin{enumerate}
		\item\label{**} Any morphism $f\colon Y\ra Z$ between rigid analytic varieties over $S$ induces a (Quillen) adjunction
		$$
		f^*\colon \bRigDA_{\et}(Z,\La)\rightleftarrows\bRigDA_{\et}(Y,\La)\colon f_*
		$$
		such that $f^*\Lambda_Z(U)=\Lambda_Y(U\times_ZY)$ for any $U\in\Sm_Z$.
		\item\label{sharp} If the morphism $f\colon Y\ra Z$  is smooth, then $f^*$ has left adjoint 
			$$
		f_\sharp\colon \bRigDA_{\et}(Y,\La)\rightleftarrows\bRigDA_{\et}(Z,\La)\colon f^*
		$$
			such that $f_\sharp\Lambda_Y(U)=\Lambda_Z(U)$ for any $U\in\Sm_Y$.
				\item\label{local} If $i\colon Z\ra S$ is a Zariski closed immersion of varieties over $S$ with open complement $j\colon U\ra S$, then the pair $(i^*,j^*)$ is conservative. More precisely, there is an exact triangle
			$$
			j_\sharp j^*\ra \id\ra i_*i^*\ra [1]
			$$
			and canonical equivalences $j^*i_*\cong0$, $i^*i_*\cong\id$.
			\item\label{jos} The functor from quasi-projective schemes over $\Spec A$ to DG-categories defined as follows $Y\mapsto \bRigDA_{\et}(Y^{\an},\La)$ has the structure of a stable homotopic  $2$-functor in the sense of \cite[D\'efinition 1.4.1]{ayoub-th1}. In particular, any morphism $f\colon Y\ra Z$ between quasi-projective schemes over $\Spec A$ induces an adjoint pair
			 $$f_!^{\an}\colon \bRigDA_{\et}(Y^{\an},\La)\rightleftarrows\bRigDA_{\et}(Z^{\an},\La)\colon f^{\an!}$$
			 with the following properties.
				\begin{enumerate}[(i)]
				\item If $f\colon W\ra Y$ is smooth of relative dimension $d$ then, locally on $W$, we have $f_!^{\an}(-)\cong f^{\an}_\sharp(-\otimes\Lambda(-d)[-2d])$. 
				\item\label{projbc} There is a natural transformation $f^{\an}_!\Rightarrow f_*^{\an}$ which is invertible if $f$ is proper.%
			\end{enumerate}
	\end{enumerate}
	\end{thm}

\begin{proof}
See (the \'etale versions of) \cite[Proposition 1.4.9, Corollaire 1.4.28, Th\'eor\`eme 1.4.33]{ayoub-rig} and \cite[Scholie 1.4.2]{ayoub-th1}.
\end{proof}

\begin{rmk}According to the structure of a stable homotopic  $2$-functor in the sense of \cite[D\'efinition 1.4.1]{ayoub-th1},  any morphism $f\colon Y\ra Z$ between quasi-projective schemes over $\Spec A$ also induces an adjoint pair
	$$f^{\an*}\colon \bRigDA_{\et}(Z^{\an},\La)\rightleftarrows\bRigDA_{\et}(Y^{\an},\La)\colon f_{*}^{\an}.$$
	These functors coincide with those induced by the map of analytic varieties $f^{\an}$ according to Theorem \ref{thm:6f1}\eqref{**}. We can therefore use the notation $(f^{\an*},f^{\an}_*)$ unambiguously. %
\end{rmk}

On the other hand we point out that the full six functor formalism for rigid analytic varieties is not yet available in the literature (progress in this direction is being done by Ayoub, Gallauer and the second named author).  For this paper the limited version stated above will be enough. More precisely, we will need the following corollary:

\begin{cor} \label{cor:proj_base_change}
	Let $f\colon W\ra S$ be a smooth map of affinoid rigid analytic varieties and let $p: P  \to S$ be a composition of a Zariski closed embedding and the canonical projection $\P^{N, \an}_S \ra S$ for some $N\in\N$. Consider the following cartesian diagram
	$$
	\xymatrix{
		P'\ar[r]^{p'}\ar[d]_{f'}&W\ar[d]^{f}\\
		P\ar[r]^{p}&S 
	}
	$$
	The natural transformation
	\[ f^*p_*\Rightarrow p'_*f'^* \]
	is invertible.	
\end{cor}
\begin{proof}
	It suffices to consider separately the case when $p$ is a Zariski closed immersion and when it is the projection $\P^N_S\ra S$ over an affinoid variety. In the latter case, we have (see Theorem \ref{thm:6f1} %
	$p_*\cong \tilde{p}^{\an}_!\cong p_\sharp(-\otimes\La(-N)[-2N])$ where 
	 $\tilde{p}$ is the projection of schemes $\P^{N}_{\tilde{S}}\ra\tilde{S}$ and where $\tilde{S}$ is $\Spec\mathcal{O}(S)$. The commutation between this functor and $f^*$ follows from the commutation of $p_\sharp$ and $f^*$ as shown below:
	$$
	{p'_*}^{} f'^*M\cong p'_\sharp(f'^*M(-N)[-2N])\cong f^*p_\sharp(M(-N)[-2N])\cong f^*p_*^{}M.
	$$

	We now suppose that $p$ is a Zariski closed embedding with complement $j\colon U\ra S$. By means of \eqref{local} of Theorem \ref{thm:6f1}  it suffices to show that the two natural transformations
	\[ p'^*f^*p_*\Rightarrow p'^*p'_*f'^*\quad j'^*f^*p_*\Rightarrow j'^*p'_*f'^* \]
	are invertible, where $j': U' \to W$ is the open complement of $P'$. Using the commutation of $f^*$ with $p'^*$ and $j'^*$ and the equivalences $p^*p_*\cong\id$, $p'^*p'_*\cong\id$, $j^*p_*\cong0$ and  $j'^*p'_*\cong 0$ the result is immediate.
\end{proof}

The following result is  an adaptation  of \cite[Proposition 4.4.3]{CD-etale} to the rigid setting and follows formally from the previous versions of the proper base change. We recall that we write simply $\iota^*$ for the derived functor that was previously denoted by $\L\iota^*$.

\begin{prop} \label{lemma:projective_base_change}
	Let $S$ be an affinoid rigid analytic variety and let $p: P \to S$ be a map which is a composition of a Zariski closed immersion and the canonical projection $\P^{N,\an}_S\ra S$. The  natural transformation
	\[  \iota^* \circ p_*\Rightarrow p_* \circ \iota^* \]
	is invertible.
\end{prop}
\begin{proof}
	We first prove an auxiliary result, namely that for any complex $L$ in $\bD(P_{\et},\Lambda)$, we have $\Hom(\Lambda,p_* \iota^* L)\cong\Hom(\Lambda,\iota^* p_* L)$.
	
	By the fully faithfulness of $\iota^*$ the second term can be re-written as
	$$
	\Hom(\Lambda,\iota^*p_*L)\cong\Hom(\Lambda,p_*L).
	$$
	On the other hand, the first term can be re-written as
	$$ \Hom(\La, p_* \iota^* L) \cong \Hom(p^* \La, \iota^* L)  \cong \Hom(p^* \iota^* \La, \iota^* L)  $$
	$$ \cong \Hom(\iota^* p^* \La, \iota^* L) \cong \Hom(p^* \La, L) \cong \Hom(\La, p_* L)$$%\cong\Hom(\iota^*\Lambda,\iota^*p_*L) $$
	where we used the adjunction $(p^*, p_*)$, the fact that $\iota^*$ is fully faithful and the fact that $p^* \iota^* \La \cong \iota^* p^* \La$ because both sheaves are isomorphic to the constant sheaf on $P$. Moreover, the natural transformation induces the identity on these groups, as wanted.
	
	We now prove the general statement. By Proposition \ref{prop:cptgen} the category $\bRigDA_{\et}(S,\Lambda)$ is generated by the motives of smooth affinoid varieties over $S$. It suffices then to prove that  for each $f\colon W\ra S$ smooth and each $L\in \bD(P_{\et},\Lambda)$ the natural transformation of the statement induces an isomorphism between the group $\Hom(f_\sharp\Lambda,p_*  \iota_*L)$ and $\Hom(f_\sharp\Lambda, \iota_*p_* L)$.
	We let $L'$ be $f'^* L$ where the map $f'$ (resp. $p'$) can be introduced as the base change of $f$ over $p$ (resp. of $p$ over $f$). By means of the projective base change for rigid motives (see Corollary  \ref{cor:proj_base_change}) we know that the former is canonically isomorphic to
	$$
	\Hom(f_\sharp\Lambda,p_*  \iota^*L)\cong \Hom(\Lambda,f^*p_*  \iota^*L)\cong\Hom(\Lambda,p'_* f'^* \iota^*L)\cong \Hom(\Lambda,p'_* \iota^* L'). 
	$$
	
	On the other hand, by means of the proper base change for Artin motives (see \cite[Proposition 4.4.1]{Huber}) we see that the latter is isomorphic to
	$$
	\Hom(f_\sharp\Lambda, \iota^*p_* L)\cong \Hom(\Lambda,f^* \iota^*p_* L)\cong \Hom(\Lambda, \iota^*f^*p_* L)\cong \Hom(\Lambda, \iota^*p'_* L').
	$$ 
	The two  are therefore canonically isomorphic by what we proved in the first part.
\end{proof}

We end our list of preliminary statements with the following  algebraic approximation result. %

\begin{prop} \label{prop:immersion}
	Let $X = \Spa A$ be a (rig-)smooth affinoid space over $\Spa R$, then there exist finite coverings $\{X_i\}$ of $X$ and $\{\Spa R_i\}$ of $\Spa R$ by admissible open affinoid subdomains such that each $X_i$ admits an embedding $j_i: X_i \to Z_i$ where $Z_i$ is the analytification of a smooth algebraic variety over $R_i$.
\end{prop}
\begin{proof}
	By   \cite[Corollaire 1.1.51]{ayoub-rig} $X$ admits a finite covering by admissible open affinoid subdomains $X_i$ each of which admits an \'etale map to $\B_{R_i}^{n_i}$ with $R_i$ as in the statement. Then, applying  \cite[Lemme 1.1.52]{ayoub-rig} to the \'etale maps $X_i \to \B_{R_i}^{n_i}$ we obtain that $X_i = \Spa(A_i)$ with $A_i$ a $K$-affinoid algebra that admits a presentation of the form
	\[ A_i = \frac{R_i \lt T_1, \ldots, T_n, U_1,\ldots,U_m \gt}{(P_1, \ldots, P_m)} \]
	with $P_j \in R_i [T_1, \ldots, T_n , U_1,\ldots,U_m]$ polynomials and $d\colonequals \det(\frac{\partial P_i}{\partial U_k})$ invertible in $A_i$.

	Therefore, we can embed $X_i$ as an admissible open affinoid of the analytification of the smooth affine $R_i$-variety
	\[ \Spec \l( \frac{R_i [ T_1, \ldots, T_n,U_1,\ldots,U_m ]}{(P_1, \ldots, P_m)}\left[\frac{1}{d}\right] \r ). \]
\end{proof}

We are finally able to prove our main result.

\begin{proof}[Proof of Theorem \ref{thm:rigmain}]
In light of the Embedding Theorems \ref{thm:embedding_DA} and \ref{thm:embDM} we are left to prove that a set of generators (as triangulated category with infinite sums) for $\bRigDA_{\et}(S,\Lambda)$ [resp. $\bRigDM_{\et}(S,\Lambda)$] lies in the image of the functors $\L \iota^*$. Since the functor 
$$\L a_{\tr}\colon\bRigDA_{\et}(S,\Lambda)\ra\bRigDM_{\et}(S,\Lambda)$$
sends a set of compact generators to a set of compact generators, we deduce that it suffices to prove the claim for  $\bRigDA_{\et}(S,\Lambda)$.

The category $\bRigDA_{\et}(S,\Lambda)$ is generated by motives of the form $ \Lambda_S(Y)$ with $Y$ affinoid and smooth over $S$. We can consider an arbitrary covering $\{S_i\}_{i \in I}$ of $S$ (resp. an arbitrary covering $\{Y_j\}_{j \in J}$ of $Y$) and write $\Lambda_S(Y)$ as the homotopy colimit of the \u{C}ech hypercover induced by the covering $\{S_i \times_S Y \}_{i \in I}$ (resp. $\{Y_j\}_{j \in J}$). Since  $\Lambda_S(S_i\times_S Y)=(j_i)_\sharp \Lambda_{S_i}(S_i\times_S Y)$, with $j_i \colon S_i\ra S$ being the open inclusion, and $(j_i)_\sharp$ sends Artin motives to Artin motives, it suffices  to prove that $\Lambda_{S}(Y)$ is Artin, locally on $S$ and $Y$.   

According to Proposition \ref{prop:immersion}, we can then suppose that 
\begin{itemize}
\item $S \cong \Spa R$ is affinoid;
\item that $f: Y = \Spa(A) \to S$ is affinoid;
\item $Y$ can be embedded as an admissible open inside the analytification of a smooth affine scheme $f'\colon Y'\ra\Spec R$ over $R$  of finite type, of pure dimension $d$.
\end{itemize}
From now on, all analytification functors are  over $\Spec R$. 
 We can embed  $Y'$ into a closed subvariety $Z$ of $\P^N_R$ by means of an open immersion $j'\colon Y'\ra Z$. We let $p$ be the structural morphism $p\colon Z\ra\Spec R$. 
The induced open immersion $j\colon Y\ra Z^{\an}$ has a formal model $\tilde{j}\colon\sY\ra\sZ$ so that $\Lambda_{Z^\an}(Y)=\xi(\tilde{j}_\sharp \Lambda_{\sY}(\sY))$, where we let $\bFormDA_{\et}(\sZ,\La)$ be the category of \'etale motives over smooth formal schemes  topologically of finite type over $\sZ$ (see \cite[Section 1.4.2]{ayoub-rig}) and
\[ \xi: \bFormDA_{\et}({\sZ}, \Lambda) \to \bRigDA_{\et}({Z^\an}, \Lambda) \]
is induced by the generic fiber functor $\sX\mapsto\sX_\eta$.  We recall that by the algebraic Rigidity Theorem \cite{ayoub-etale} and \cite[Corollary 1.4.23]{ayoub-rig} the special fiber functor $\sX\mapsto\sX_\sigma$ determines an equivalence $$
\bFormDA_{\et}({\sZ}, \Lambda)\cong\bDA(\sZ_\sigma,\La)\cong\bD((\sZ_{\sigma})_{\et},\La)$$
By  the canonical equivalence (see \cite[Th\'eor\`eme 18.1.2]{EGAIV4}) $\bD(\sZ_{\et},\Lambda)\cong\bD((\sZ_\sigma)_{\et},\Lambda)$ we  deduce that $\tilde{j}_\sharp \Lambda_{\sY}(\sY)$ is an Artin motive. Since $\xi$ preserves Artin motives (the generic fiber of an \'etale scheme over $\sS$ is \'etale over $\sS_\eta$), we deduce that $\Lambda_{Z^\an}(Y)$ is Artin, equal, say, to $\iota^*N$ with $N\in\bD(Z^{\an},\Lambda)$.

Since the analytification of an \'etale extension is still \'etale, the functor $\An^*$ preserves Artin motives. Also the functor $j'_\sharp$ does by its explicit description (Theorem \ref{thm:6f1}\eqref{sharp}). We   deduce then that $\An^*(j'_\sharp(\Lambda_{Y'}(Y')(d))[2d])$
  is $\iota^* M$ for some Artin motive $M$ over $Y'^{\an}$.

We let $\alpha$ be the immersion $Y\ra Y'^{\an}$. 
We now give an alternative description of the motive $\Lambda_S(Y)$. By Theorem \ref{thm:6f1} (1) and (2)  we have that ${f'_!}^{\an}\cong p_!^{\an}{j'_!}^{\an}\cong p_*^{\an}{j'}_\sharp^{\an}$ and up to considering  coverings of $S$ and $Y$, we can suppose that $f'^{\an}_\sharp(\alpha_\sharp \Lambda_Y(Y))\cong f'^{\an}_!(\alpha_\sharp \Lambda_Y(Y)\otimes\Lambda(d)[2d])$.

We then deduce 
$$
\begin{aligned}
\Lambda_S(Y)&\cong f_\sharp\Lambda_Y(Y)\cong{ f'^\an}_\sharp\alpha_\sharp\Lambda_Y(Y)\cong f'^{\an}_!(\alpha_\sharp\Lambda_Y(Y)\otimes\Lambda_{Y'^{\an}}(Y'^{\an})(d)[2d])\\ &\cong p^{\an}_*j'^{\an}_\sharp(\alpha_\sharp\Lambda_Y(Y)\otimes\Lambda_{Y'^{\an}}(Y'^{\an})(d)[2d])\\
&\cong p^{\an}_*(j_\sharp \La_Y(Y)\otimes j'^{\an}_\sharp(\Lambda_{Y'^{\an}}(Y'^{\an})(d))[2d])\\&\cong p^{\an}_*(\iota^*N\otimes \An^*(j'_\sharp(\Lambda_{Y'}(Y')(d))[2d]))\\&\cong p^{\an}_*\iota^*(N\otimes M)
\end{aligned}$$
where  we used the fact that the functor $j'^{\an}_\sharp$ commutes with tensor products of effective motives (see its definition - Theorem \ref{thm:6f1}\eqref{sharp}) together with the equality $ j=j'^{\an}\circ\alpha$ and, %
for the last equality, the monoidality of the functor $\iota^*$ (the product in $S_{\et}$ is the same as the one in $\Sm_S$). 
In order to prove that $\Lambda_S(Y)$ is Artin it suffices then to show that $p_*^{\an}$ commutes with $\iota^*$ which is proven in Proposition \ref{lemma:projective_base_change}.
\end{proof}

\section{Applications} \label{sec:applications}

In this concluding section we describe three applications of our main theorems. We keep the hypothesis used so far, with the difference that now $\La$ is supposed to be any ring. We keep the notation that $K$ is a complete non-Archimedean non-trivially valued field with finite $\ell$-cohomological dimension, and $S$ a normal rigid analytic variety over it.

\subsection{The \'etale realization functor}
We follow \cite{ayoub-etale} and we show how to use the Rigidity Theorem to produce \'etale realization functors. In particular, we can produce an $\ell$-adic realization functor for the category $\bRigDA_{\et}(K,\Q)$ completing the picture of ``classical'' realizations for rigid analytic motives (for a $p$-adic realization for this category see \cite{vezz-mw} and for a Betti-like realization see \cite{vezz-berk}).

We first introduce the target category, which is the derived category of $\ell$-adic \'etale sheaves, following Ekedhal. We refer to \cite[Section 5]{ayoub-etale} for details on this construction, which we simply adapt to the rigid analytic situation. 

\begin{defn}%
Let $J \subset \La$ be an ideal. The objects of the category $\La/J^*\Mod$  are diagrams of $\La$-modules
    $$
M_\bullet=\cdots\ra M_{s+1}\ra M_s\ra\cdots\ra M_1\ra M_0
    $$
    such that $J^s\cdot M_s=0$ for all $s\in\N$. In particular, each $M_s$ is canonically a $\La/J^s$-module. Morphisms are defined level-wise. The category $\Ch(\La/J^*\Mod)$ can be endowed with a (projective) model structure, giving rise to the category $\bD(S_{\et},\La/J^*)$ and  $\bRigDA_{\et}(S,\La/J^*)$ defined as in Definition \ref{defn:etale_motives}. The functors $M_\bullet\mapsto M_s$ induce left Quillen functors $$s^*\colon \bRigDA_{\et}(S,\La/J^*)\ra\bRigDA_{\et}(S,\La/J^s)$$ 
    $$s^*\colon \bD(S_{\et},\La/J^*)\ra\bD(S_{\et},\La/J^s)$$
    which are jointly conservative by (the obvious analogues of) \cite[Lemme 5.3 and 5.4]{ayoub-etale}. 
    
    We define the category 
$\widehat{\bD}(S_{\et},\La_J)$ to be the full triangulated subcategory of $ \bD(S_{\et},\La/J^*)$ whose objects are complexes $K$ such that the canonical map $(s+1)^*K\otimes_{\La/J^{s+1}}\La/J^s\ra s^*K$ is invertible for each $s\in\N$.

    We define the category 
$\widehat{\bD}^{\mathrm{ct}}(S_{\et},\La_J)$ (where $\mathrm{ct}$ stands for \emph{constructible}) to be the full triangulated subcategory of $ \widehat{\bD}(S_{\et},\La_J)$ whose objects are complexes $K$ such that for each $s\in\N$, 
the \'etale sheaves associated to $H_n(s^*K)$ are constructible in the sense of \cite[Definition 2.7.2]{Huber} for all $n\in\Z$, and zero for $|n|\gg0$.  
 
We define the category 
$\widehat{\bD}^{\mathrm{ct}}(S_{\et},\La_J\otimes\Q)$ to be the triangulated category %
 obtained from $\widehat{\bD}^{\mathrm{ct}}(S_{\et},\La_J)$ by tensoring all $\Hom$-groups with $\Q$ (it is still a pseudo-abelian  triangulated category, see \cite[Corollary B.2.3]{CD-etale}).%
\end{defn}

We can make a direct-limit argument on coefficients and use Rigidity to produce $\ell$-adic realizations. For a triangulated category $\bT$ we denote by $\bT^{\comp}$ the full subcategory of compact objects.
 
\begin{thm}\label{thm:etre}
Let $J$ be an ideal of $\La$ for which $\La/J$ is $N$-torsion, with $N$ invertible in the residue field of $K$. For any normal rigid analytic variety $S$ over $K$ there are  monoidal triangulated functors
$$
\fR_{S,J}\colon\bRigDA_{\et}(S,\La)\ra\widehat{\bD}(S_{\et},\Lambda_J)
$$
$$
\fR_{S,J}^{\Q}\colon\bRigDA^{\comp}_{\et}(S,\La\otimes\Q)\ra\widehat{\bD}^{\mathrm{ct}}(S_{\et},\Lambda_J\otimes\Q)
$$
which have the following properties (for the six functors formalism, see  Theorem \ref{thm:6f1} for the motivic side, and \cite{Huber} for the \'etale side).
\begin{enumerate}
	\item\label{etre} The composition of the first with $\widehat{\bD}(S_{\et},\Lambda_J)\ra \bD(S_{\et},\Lambda/J^s)$ gives the functor $$ \bRigDA_{\et}(S,\La)\stackrel{(-)\otimes^{\L}\La/J^s}{\longrightarrow} \bRigDA_{\et}(S,\La/J^s)\stackrel{\R\iota_*}{\cong} {\bD}(S_{\et},\Lambda/J^s)$$
	where the second equivalence is the one from the Rigidity Theorem.
\item They commute with the functors $f_*$ and $f^*$, for any morphism $f$ of normal rigid-analytic varieties. 
\item They commute with the functors $f_\sharp$ (left adjoint to $f^*$ in case $f$ is smooth).
\item They commute with the functors $f^{\an}_!$ and $f_{\an}^!$ for any morphism $f$ of quasi-projective schemes over a Tate algebra.
\item They commute with the bi-functor $\R\Hom$ when restricted to compact objects in the first variable.
\end{enumerate}
\end{thm}
\begin{proof}
For the functor $\fR_{S,J}$, it suffices to copy verbatim the arguments of \cite[Th\'eor\`eme 6.9]{ayoub-etale}. Indeed, Rigidity gives rise to a canonical equivalence of tensor DG-categories $\bRigDA_{\et}(S,\La/J^*)\cong\bD(S_{\et},\La/J^*)$. The realization functor is then defined by means of
	$$
	\bRigDA_{\et}(S,\La)\ra	\bRigDA_{\et}(S,\La/J^*)\cong\bD(S_{\et},\La/J^*)
	$$
	where the first functor is induced by the obvious functor $M\mapsto (M\otimes\La/J^*)$ from $\La\Mod$ to $\La/J^*\Mod$. It is easily seen to take values in $\widehat{\bD}(S_{\et},\Lambda_J)$ and to commute with $f_\sharp$, $f^*$. As for $f_*$ one can use the proof of \cite[Th\'eor\`eme 6.3]{ayoub-etale} while for $\R\Hom$  one can use the proof of \cite[Th\'eor\`eme 6.4]{ayoub-etale}. The point $(4)$ deals with algebraic morphisms, so the statement for $f^{\an}_!$ follows at once from \cite[Th\'eor\`eme 1.4.33]{ayoub-rig} and \cite[Th\'eor\`eme 3.4(b)]{ayoub-betti} while for $f^!_{\an}$ one can consider separately the case where $f$ is smooth (in which \cite[Th\'eor\`eme 3.4(c)]{ayoub-betti}  holds) and the case where $f$ is a closed immersion, which follows from the localization triangle \cite[Proposition 1.4.9]{ayoub-th1} and point (2).   We leave the details to the reader.

    Following \cite[Th\'eor\`eme 9.7]{ayoub-etale}, in order to extend the functor $\fR_{S,J}$ (and its properties) to the functor $\fR_{S,J}^{\Q}$ one needs to prove that $\fR_{S,J}$ sends compact objects to the category $\widehat{\bD}^{\mathrm{ct}}(S_{\et},\La_J)$. As already remarked in Proposition \ref{prop:immersion}, a class of compact generators for $\bRigDA_{\et}(S,\La)$ is obtained by $f_\sharp(\La)(n)$ with $f\colon X\ra \Spa A\subset S$ is a smooth map of an affinoid  variety $X$ over an affinoid open subvariety $\Spa A$ of $S$. We can also suppose that $X$ is open in  a smooth variety $j\colon X\subset Y^{\an}$ which is the analytification of a smooth $A$-scheme of finite type $g\colon Y\ra\Spec A$. By the commutation of $\fR_{S,J}$ with the six operations and Theorem \ref{thm:6f1} we obtain $\fR_{Y^{\an},J} j_\sharp\cong j_!\fR_{X,J} $ and $\fR_{\Spa A,J}g^{\an}_\sharp\cong  g_!^{\an}(\fR_{Y^{\an},J}(n)[2n])$ %
     (the last isomorphism holding locally on $X$, see  Theorem \ref{thm:6f1}(4i)). It therefore suffices to show that for a smooth morphism $f\colon X\ra S$ of rigid analytic varieties, the functor $f_!\colon\widehat{\bD}(X_{\et},\La/J^s) \ra\widehat{\bD}(S_{\et},\La/J^s) $ preserves constructible complexes, which follows from \cite[Theorem 6.2.2]{Huber}. 
\end{proof}

\begin{rmk}Suppose $\La=\Z$ and $J=(\ell)$ with $\ell$ a prime invertible in the residue field of $K$. 
We point out that the realization functor $\fR$ computes $\ell$-adic \mbox{(co-)homology}. In order to do so we first put $\La'=\Z/\ell^N$ and compute, for any $p\colon X\ra S$ in $\Sm_S$ using Rigidity:
$$
\R^n p_*\La'\cong H_n\R\Hom(\La'_S(X),\L\iota^*\La')\cong H_n\R\Hom(\R\iota_{*}\La'_S(X),\La')%
$$
We deduce, by Theorem \ref{thm:etre}\eqref{etre}, that $\R\Hom(\fR_{S,\ell}\Z_S(X),\Z_\ell)\cong\R p_*\Z_\ell$.
\end{rmk}

\begin{rmk}
	We point out that the $\ell$-adic realization can alternatively be constructed using the universal property of rigid motives and the pro-\'etale cohomology \cite{bs,scholze-ph}. This is particularly convenient in order to define an infinity-category $\bD(S_{\proet},\Lambda\otimes\Q)$ containing $\widehat{\bD}(S_{\et},\Lambda_J\otimes\Q)$ as a full subcategory, allowing one to extend the functor $
	\fR_{S,J}^{\Q}$ 
	to a functor
	$$
	\fR_{S,J}^{\Q}\colon\bRigDA_{\et}(S,\La\otimes\Q)\ra\bD(S_{\proet},\Lambda\otimes\Q)
	$$ 
	by means of the universal property of $\bRigDA_{\et}(S,\La\otimes\Q)$. To this aim, it suffices to define a functor from the category $\Sm_S$ of smooth affinoid varieties over $S$ to $\bD(S_{\proet},\Lambda\otimes\Q)  $ having \'etale descent and homotopy invariance. One can then take the functor $X\mapsto\fR_{S,J}^{\Q}(\Lambda_S(X))$ defined above.
\end{rmk}

\begin{rmk}In order to prove that $\fR_{S,J}$ sends compact objects to constructible complexes, one could also argue by considering the composition with each $s^*$ which is canonically equivalent to the change of coefficients $\bRigDA_{\et}(S,\La)\ra\bRigDA_{\et}(S,\La/J^s)$ (see Theorem \ref{thm:etre}\eqref{etre}). This functor has a right adjoint which commutes with sums, hence it preserves compact objects. The result then follows from the comparison between compact and constructible complexes see \cite[Proposition 20.17]{scholze-diam}.
\end{rmk}

In case the base variety is the analytification of a quasi-projective scheme ${S}$ over a Tate algebra $A$, one can consider the analytification functor between algebraic and analytic motives. It is known that the \'etale cohomology
of $S$ coincides with the one of ${S^{\an}}$. We have the following motivic strengthening of this fact.

\begin{cor}
	Let  ${S}$ be a normal scheme of finite presentation over a Tate algebra  over $K$ and let $\La$ and $J$ be as in Theorem \ref{thm:etre}.  The analytification functor $$\L\An^*\colon\bDA_{\et}(S,\La) \ra\bRigDA_{\et}(S^{\an},\La)$$
	is compatible with the \'etale realization functors, i.e. there is a commutative diagram
	$$
	\xymatrix{
	\bDA_{\et}({S},\La)\ar[r]^-{\L\An^*}\ar[d]^{\fR_{S,J}} &
	\bRigDA_{\et}(S^{\an},\La)\ar[d]^-{\fR_{S^{\an},J}} \\
	\widehat{\bD}({S_{\et}},\La_J)\ar[r]^{\L\An^*} &
\widehat{\bD}({S^{\an}_{\et}},\La_J)\ \\
}
	$$
\end{cor}

\begin{proof}
	By construction, it suffices to prove that the analytification functor is compatible with the functors arising with the change of coefficients $\La\ra\La/\ell^N$ which is obvious.
\end{proof}

\begin{rmk}
	By putting $S=\Sp K$ the Rigidity Theorem gives  equivalences
	$$\bD(K_{\et},\La)\cong\bDA(K,\La)\cong\bRigDA_{\et}(K,\La).
	$$
	In particular, $\ell$-adic \'etale cohomology is insensitive to analytification, which is already shown in \cite[Corollary 3.8.1]{Huber}.
\end{rmk}

\subsection{Rigid motives with and without transfers}

The Rigidity Theorem permits to improve the known comparison results about the categories $\bRigDA_{\et}(K, \La)$ and $\bRigDM_{\et}(K, \La)$, similarly to the algebraic case (see \cite[Annexe B]{ayoub-etale}). The current state of the art in the setting of rigid analytic motives is the following.

\begin{thm} \label{thm:comparison_ayoub}
	Let $\La$ be a $\Q$-algebra. The canonical functor 
	$$\L a_{\tr}\colon\bRigDA_{\et}(K,\Lambda)\ra\bRigDM_{\et}(K,\Lambda)$$
	is an  equivalence of monoidal DG-categories.
\end{thm}
\begin{proof}
By \cite[Corollary 4.20]{vezz-DADM} we can deduce the statement for the categories $\bRigDM_{\et}(K,\Lambda)$ and $\bRigDA_{\Frobet}(K,\Lambda)$ the latter being the localization of $\bRigDA_{\et}(K,\Lambda)$ over the relative Frobenius maps (see \cite[Section 2.1]{vezz-DADM} for the definition of the Frob-\'etale topology and its relation with motives). We claim that the relative Frobenius maps are already equivalences in $\bRigDA_{\et}(K,\Lambda)$. Indeed, they induce an endofunctor $\Phi$ and a natural transformation $\Phi\Rightarrow \id$ in $\bRigDA_{\et}(K,\Lambda)$. We let $\catT$ be the full triangulated subcategory of those objects where the natural transformation is invertible. By the separatedness property of $\bDA_{\et}(K,\Lambda)$ (see \cite[Th\'eor\`eme 3.9]{ayoub-etale})  the category $\catT$, which is obviously closed under sums and cones, contains the motives of the form $\L\An^*M$ with $M\in \bDA_{\et}(K,\Lambda)$. But such motives generate the whole category (by \cite[Th\'eor\`eme 2.5.35]{ayoub-rig}) so that $\catT=\bRigDA_{\et}(K,\Lambda)$ as wanted.
\end{proof}

Using the same strategy as in \cite{ayoub-etale} we can promote the comparison above to motives with $\Z[1/p]$-coefficients using the Rigidity Theorem. 

\begin{thm} \label{thm:comparison}
		Let $\La$ be a $\Z[\frac{1}{p}]$-algebra, where $p$ is the exponential residual characteristic of $K$.  The canonical functor 
	$$\L a_{\tr}\colon\bRigDA_{\et}(K,\Lambda)\ra\bRigDM_{\et}(K,\Lambda)$$
	is an  equivalence of monoidal DG-categories.
\end{thm}
\begin{proof}
	It suffices to adapt the proof of \cite[Corollaire B.3]{ayoub-etale} to rigid motives. We recall this argument for the convenience of the reader.
	
	We remark that the functor $\L a_{\tr}\colon\bRigDA_{\et}(K,\Lambda)\ra\bRigDM_{\et}(K,\Lambda)$ sends a generating set of compact objects of the first category (motives of the form $\Lambda(S)[n](k)$ with $S$ affinoid, \'etale  over $\B^N_K$) to a generating set of compact objects of the second (motives $\Lambda^{\Tr}(S)[n](k)$ with $S$ as before). By means of \cite{ayoub-rig}[Lemme 1.3.32] it then suffices to show that for any such $M,N$ compact in $\bRigDA_{\et}(K,\Lambda)$ the map
	\begin{equation}\label{bij}
	\Hom_{\bRigDA_{\et}(K,\Lambda)}(M,N)\ra\Hom_{\bRigDM_{\et}(K,\Lambda)} (\L a_{\Tr}M, \L a_{\Tr}N)
	\end{equation}
	is bijective. We remark that for any map of rings $\Lambda\ra\Lambda'$ and any object $N$ of $\bRigDA_{\et}(K,\Lambda')$ (resp. $\bRigDM_{\et}(K,\Lambda)$) we have  the following adjunctions
	$$
	\Hom_{\bRigDA_{\et}(K,\Lambda')}(M\otimes^{\L}_{\Lambda}\Lambda',N)\cong\Hom_{\bRigDA_{\et}(K,\Lambda)}(M,N)
	$$
	$$
	\Hom_{\bRigDM_{\et}(K,\Lambda')}(M\otimes^{\L}_{\Lambda}\Lambda',N)\cong\Hom_{\bRigDM_{\et}(K,\Lambda)}(M,N)
	$$
	and since we may assume that $M=\Lambda(S)[n](k)\cong\Z[1/p](S)[n](k)\otimes^{\L}\Lambda$ it suffices to consider the case $\Lambda=\Z[1/p]$.
	
	We can consider the exact triangle
	$$
	\Z[1/p]\ra\Q\ra\Q/\Z[1/p]\cong\colim\Z/\ell^k\Z\ra(\Z[1/p])[1]
	$$
	and tensor it with $N$, where the colimit ranges over primes $\ell\neq p$ and natural numbers $k$. Since $M$ as well as $\L a_{\Tr}M$  are compact, we are reduced to proving the bijectivity of  \eqref{bij} for the cases $N=N\otimes\Q$ and $N=N\otimes\Z/\ell^k$ separately. Using the adjunctions above once more, one can test the fully faithfulness of $\L a_{\Tr}$ in the cases $\Lambda=\Q$ and $\Lambda=\Z/\ell^k$ which follow from Theorem \ref{thm:comparison_ayoub} and Theorem \ref{thm:rigmain} respectively. 
\end{proof}

\begin{rmk}
Even though there is a version of Theorem \ref{thm:comparison_ayoub} which holds also for effective motives (see \cite[Theorem 4.1]{vezz-DADM})  there is no effective version of the Rigidity Theorem without transfers, so that Theorem \ref{thm:comparison} cannot be stated for effective motives.
\end{rmk}

\subsection{The motivic tilting equivalence with $\Z[1/p]$-coefficients}

In this section we focus on rigid motives over perfectoid fields and we answer positively to the conjecture stated in \cite[Remark 7.28]{vezz-fw}. We recall that a perfectoid field $K$ is a non-Archimedean field such that $|p| < 1$ for a prime $p$ and that the Frobenius morphism $\cO_K/p \to \cO_K/p$ is surjective (classical examples include $\C_p$ or the completion of $\Q_p(\mu_{p^\infty})$). The interest about these objects comes from the tilting functor $K \mapsto K^\flat$  which associates to $K$ a perfectoid field of characteristic $p$ (in the examples above, one obtains the complete algebraic closure of $\F_p(\!(T)\!)$ and the completion of $\F_p(\!(T^{1/p^\infty})\!)$ respectively). %

\begin{thm}[{\cite[Theorem 1.3]{scholze}}] \label{thm:scholze}
The tilting functor induces an equivalence between the small \'etale site over $K$ and the small \'etale site over $K^\flat$. 
\end{thm}
It can be restated motivically as an equivalence between Artin motives over the two fields. 
The main result of \cite{vezz-fw} is the extension of this result to a motivic tilting equivalence, where we write  $\bRigDA_{\et}(K,\Lambda)$ for $\bRigDA_{\et}(\Sp K,\Lambda)$.

\begin{thm} \label{thm:main_alberto}\label{thm:tilt}
	Let $\La$ be a $\Q$-algebra. There is an  equivalence
	\[ \bRigDA_{\et}(K,\La) \cong \bRigDA_{\et}(K^\flat,\La), \]
\end{thm}
\begin{proof}
	\cf \cite[Theorem 7.26]{vezz-fw} paired up with \cite[Corollary 4.20]{vezz-DADM}. See the proof Theorem \ref{thm:comparison_ayoub} on how to avoid the $\Frob$-localization. 
\end{proof}

We refer to \cite{vezz-fw} for a thorough discussion of this result and its meaning. We summarize here the construction of the connecting functor.    %

One can consider  the category of perfectoid motives $\bPerfDA_{\et}(K)$ which is obtained by  considering motives over the \'etale site of (geometrically) smooth perfectoid spaces over $K$ (homotopies are defined over the perfectoid disc $\widehat{\B}^1=\Spa K\langle T^{1/p^\infty}\rangle$ and the twist is defined by considering the cokernel of the map of sheaves $\Lambda\ra\Lambda(\Spa K\langle T^{\pm1/p^\infty}\rangle)$).

Scholze's tilting equivalence \cite[Theorems 1.9 and 1.11]{scholze}  straightforwardly implies that the adjunction induced by the tilting functor 
$$
\L\flat^*\colon \bPerfDA_{\et}(K,\Lambda)\leftrightarrows\bPerfDA_{\et}(K^\flat,\Lambda)\colon\R\flat_*
$$
is an equivalence (for any $\Lambda$).

In the case when $\car K=p$, the perfection functor induces an adjunction
\begin{equation}\label{eq:Perf}
\L\Perf^*\colon \bRigDA^{}_{\et}(K,\Lambda)\leftrightarrows\bPerfDA_{\et}(K,\Lambda)\colon\R\Perf_*
\end{equation}
which is shown to be invertible if $\Lambda$ is a $\Q$-algebra (see \cite[Theorem 6.9]{vezz-fw}  
 and the proof of Theorem \ref{thm:comparison_ayoub} to deduce the stable statement without $\Frob$-localization).

Suppose now $\car K=0$. As a further auxiliary category, one can also introduce $\wRigDA_{\et}(K,\Lambda)$ which is the motivic category over the \'etale site of ``semi-perfectoid spaces'', that is: those adic spaces which are locally \'etale over $\B^n\times\widehat{\B}^m$  (homotopies are considered over $\widehat{\B}^1$ and the twist is defined by considering the cokernel of the map of sheaves $\Lambda\ra\Lambda(\Spa K\langle T^{\pm1}\rangle)$ (we follow here the notation of \cite{vezz-tilt4rig}. This category is denoted $\widehat{\mathbf{Rig}}\bDA_{\et,\widehat{\B}^1}(K,\La)$ in \cite{vezz-fw}).
The canonical inclusion of \'etale sites induces a Quillen adjunction:
$$
\L\iota^*\colon \bRigDA_{\et}(K,\Lambda)\leftrightarrows\wRigDA_{\et}(K,\Lambda)\colon\R\iota_*
$$
and one can produce explicitly a functor (see \cite[Proposition 7.22]{vezz-fw}) %
$$
\R j_*\colon \wRigDA^{}_{\et}(K,\Lambda)\ra \bPerfDA^{}_{\et}(K,\Lambda).
$$
that can be described as follows: it maps a spectrum $\{M_i\}_{i \in \N}$ to the spectrum $\{j_*M_i\}_{i \in \N}$ where $(j^*,j_*)$ is the Quillen equivalence on effective motives induced by the inclusion of sites from perfectoid spaces in semi-perfectoid spaces. 

In particular, even in characteristic $0$, one has a connecting functor 
\begin{equation}\label{eq:ji}
\R j_*\L\iota^*\colon\bRigDA^{}_{\et}(K,\Lambda)\ra \bPerfDA^{}_{\et}(K,\Lambda)
\end{equation}
which is shown  to be invertible if $\Lambda$ is a $\Q$-algebra (see \cite{vezz-fw}).

The contribution of this paper to this topic is the following theorem.

\begin{thm} \label{thm:perfectoid_comparison}
	Let $\La$ be a $\Z[\frac{1}{p}]$-algebra, where $p$ is the residual exponential characteristic of $K$. The functor $\R\Perf_*\L\flat^*\R j_*\L\iota^*$ induces an equivalence

	\[ \bRigDA_{\et}(K,\Lambda) \cong \bRigDA_{\et}(K^\flat,\Lambda). \]
	Moreover, the equivalence above is compatible with the $\ell$-adic \'etale realization functors of Theorem \ref{thm:etre}.

\end{thm}
\begin{proof}The proof is divided in several steps.%

Step 1.	We first claim that the perfection functor \eqref{eq:Perf} is an equivalence. Indeed, the functor $\L\Perf^*$ sends a set of compact generators to a set of compact generators (see \cite[Proposition 3.30]{vezz-fw}). We can therefore argue like in the proof of Theorem \ref{thm:comparison} (see also \cite[Corollaire B.3]{ayoub-etale}) and consider separately the case $\Lambda=\Q$ and $\Lambda=\Z/\ell^N$ with $\ell$ prime different from $p$. The first case is known to be invertible (see \cite[Theorem 6.9]{vezz-fw}). We can assume now that $\Lambda$ is a torsion ring. The isomorphism $\Lambda(1)\cong\mu_{S, \ell^N}$ holds in $\bRigDA_{\et}(K^\flat,\Lambda)\cong\catD(K^\flat_{\et},\Lambda)$ and $\L\Perf^*$ induces the analogue isomorphism in $\bPerfDA_{\et}(K^\flat,\Lambda)$. As in the case of Theorem \ref{thm:embedding_DA}, one can show that $\catD(K_{\et}^\flat,\Lambda)\ra \bPerfDA_{\et}(K^\flat,\Lambda)$ is fully faithful, proving the claim.

Step 2. We now assume $\car K=0$ and we prove that $F\colonequals \R j_*\L\iota^*$ is also an equivalence. We cannot argue as in the previous step since it is not clear that the functor $F$ sends compact objects to compact objects. On the other hand, we remark that it commutes with small sums (see \cite[Remark 7.23]{vezz-fw}) 
 so that, by Brown's representability theorem, it has a triangulated right adjoint $G$. As in the proof of Theorem \ref{thm:comparison} we may and do suppose that $\La=\Z[1/p]$.  We use the same letters $F$ and $G$ to indicate these functors defined when $\Lambda=\Q$ and $\Lambda=\Z/\ell$ in which cases we know they are equivalences of categories (by \cite[Theorem 7.10]{vezz-fw} and by Rigidity together with Step 1, respectively).

Step 3. By (the obvious analogues of) \cite[Proposition 5.4.3, Paragraph 5.4.4, Proposition 5.4.12]{CD-etale} we know the following facts:
\begin{enumerate}
\item\label{cons} Each (Quillen) functor induced by  change of coefficients $\rho_\Q^*\colon\bRigDA_{\et}(K,\Z[1/p])\ra\bRigDA_{\et}(K,\Q)$ and $\rho_\ell^*\colon\bRigDA_{\et}(K,\Z[1/p])\ra\bRigDA_{\et}(K,\Z/\ell)$ with $\ell\neq p$ prime  has a conservative right adjoint $\rho_{\Q*} $ resp. $\rho_{\ell*} $ and the set $\{\rho^*_\Q,\rho^*_\ell\}_{\ell\neq p}$ is a conservative family.
\item \label{tens}$\rho_{\ell*} \rho^*_{\ell}M\cong M\otimes_{\Z} \Z/l$ and $\rho_{\Q*} \rho^*_{\Q}M\cong M\otimes_{\Z} \Q$.
\item \label{ell}$ M\otimes_{\Z} \Z/l\cong Cone(M\stackrel{\times\ell}{\ra}M)$. 
\item\label{Q} If $A$ is compact, then $\Hom(A,M\otimes\Q)\cong\Hom(A,M)\otimes\Q$.
\end{enumerate}We now prove that $F$ and $G$ commute with the functors of  change of coefficients  $\rho^*$ (in both cases $\Q$ or $\Z/\ell$).  From the conservativity stated in \eqref{cons} one deduces the following: if $G$ commutes with $\rho_{*}$ (which amounts to say that $F$ commutes with $\rho^*$)  then it commutes with $\rho^*$ if and only if it commutes with $\rho_*\rho^*$ which is $-\otimes\Q$ or $-\otimes\Z/\ell$ respectively (property \eqref{tens}). 
From the property \eqref{ell} any triangulated additive functor (like $G$) commutes with $-\otimes\Z/\ell$. Also, in order to show $G(S\otimes\Q)\cong GS\otimes \Q$ it suffices to check that for any compact object $A$ one has $\Hom(A,G(S\otimes\Q))\cong \Hom(A,GS\otimes \Q)$ and this follows from adjunction and the property \eqref{Q}. It then suffices to show that $F$ commutes with $\rho^*$. We recall that $F$ is the composition of $\L\iota^*$ (which obviously commutes with $\rho^*$) and  $\R j_*$ which is termwise defined by means of $j_*$ which is the  effective Quillen right adjoint induced by the inclusion of perfectoid spaces in semi-perfectoid ones. We are left to prove that this functor commutes with $\rho^*$. Since its left adjoint obviously does, we deduce the claim by the adjunction arguments above. 

Step 4. We can finally prove the statement of the theorem. Fix $S$ in $\bRigDA_{\et}(K,\Z[1/p])$. We want to prove that $FGS\cong S\cong GFS$. By the property \eqref{cons} it suffices to show this after applying $\rho_\Q^*$ and $\rho_{\ell}^*$. Since $F$ and $G$ commute with them as shown in Step 3, the claim follows from \cite[Theorem 7.10]{vezz-fw} and Rigidity paired up with Step 1. We remark that along the proof we have shown the compatibility of $F$ with the functor $\rho^*$ over $\Z/\ell^N$. By construction, this is enough to prove that $F$ is compatible with the $\ell$-adic realization functors.
\end{proof}

\begin{rmk}Let $\La$ be a $\Z[\frac{1}{p}]$-algebra, where $p$ is the residual exponential characteristic of $K$. We remark that there is also an \emph{effective} version of the motivic tilting equivalence with transfers (see \cite{vezz-fw}). Once it is paired up with the Cancellation Theorem \cite[Théorème 2.5.38]{ayoub-rig} and Rigidity, it enables one to  prove that the equivalence of Theorem \ref{thm:perfectoid_comparison} (which can also be stated for $\bRigDM$ using Theorem \ref{thm:comparison}) restricts to an equivalence:	\[ \bRigDM^{\eff}_{\et}(K,\Lambda) \cong \bRigDM^{\eff}_{\et}(K^\flat,\Lambda)\]
which	is the effective version of the previous theorem (with transfers).
\end{rmk}

We recall that the fields $K$ and $K^\flat$ have the same residue field $k$. In particular, for both of them we have a canonical Quillen adjunction 
$$
\xi\colon\bDA_{\et}(k,\La)\leftrightarrows \bRigDA_{\et}(K,\La)\colon\chi
$$
$$
\xi^\flat\colon\bDA_{\et}(k,\La)\leftrightarrows \bRigDA_{\et}(K^\flat,\La)\colon\chi^\flat
$$
It is obtained in the following way: first, we can consider the following equivalence, induced by the special fiber functor $\fX\mapsto \fX_k$ (see \cite[Corollaire 1.4.24]{ayoub-rig}):
$$
\bDA_{\et}(k,\La)\cong\bFormDA_{\et}(\cO_K,\La)
$$
where we let $\bFormDA_{\et}(K,\La)$ be the category of \'etale motives of smooth formal schemes topologically of finite type over $\cO_K$. Then, we can consider the (Quillen) adjunction
$$
\L(-)^*_\eta\colon\bFormDA_{\et}(\cO_K,\La)\leftrightarrows \bRigDA_{\et}(K,\La)\colon\R(-)_{\eta*}
$$
induced by generic fiber functor $\fX\mapsto \fX_\eta$. In particular $\xi(\Lambda(\fX_k))\cong\Lambda(\fX_\eta)$ for any smooth formal scheme topologically of finite type $\fX$ over $\cO_K$.

In \cite{vezz-tilt4rig} it is shown that these functors commute with the motivic tilting equivalence, whenever $\Q\subset\La$. With Rigidity, we can now complement this result for any $\La$ where $p$ is invertible.

\begin{cor}\label{cor:tiltandxi}
	Let $K$ be a perfectoid field of residual  characteristic $p$ and let $\La$ be a $\Z[1/p]$-algebra. Up to a natural transformation, the following diagram of monoidal DG-categories 	is commutative.
	\[
	\xymatrix{
		&\bDA_{\et}(k,\La)\ar[dl]_{\xi}\ar[dr]^{\xi^\flat}\\
		\bRigDA_{\et}(K,\La) %
			\ar@{<->}[rr]^{\sim}&&\bRigDA_{\et}(K^\flat) 
		}
		\] 
	\end{cor}

\begin{proof}
We can assume $\car K=0$ and we follow the notation and the proof of \cite[Theorem 3.6]{vezz-tilt4rig}. In particular we use the notion of perfectoid space over $k$ introduced in Section 2 of \cite{vezz-tilt4rig} and the relative category of motives.
With no hypotheses on $\La$, one can define a natural transformation $ \xi\circ\R j_*\circ\L\iota^*\Rightarrow\R j_*\circ\L\iota^*\circ\xi$ between the two sides of the following square
	$$\xymatrix{
		\bDA_{\et}(k,\Lambda)\ar[d]_{\xi}\ar[r]^-{\R j_*\L\iota^*}&	\bPerfDA(k,\Lambda)\ar[d]_{\xi}\\
		\bRigDA_{\et}(K,\Lambda)	\ar[r]^-{\R j_*\L\iota^*}	&\bPerfDA(K,\Lambda)\\	
	}
	$$
	and an invertible natural transformation $ \xi\circ\L\Perf^*\cong\R\flat_*\circ\L\Perf^*\circ\xi^\flat$ between the two sides of the following square
			$$\xymatrix{
		\bDA_{\et}(k,\Lambda)\ar[d]_{\xi^\flat}\ar[rr]^-{\L\Perf^*}& &\bPerfDA(k,\Lambda)\ar[d]_{\xi}\\
		\bRigDA_{\et}(K^\flat,\Lambda)	\ar[r]^-{\L\Perf^*}&	\bPerfDA(K^\flat,\Lambda) \ar[r]^-{\R\flat_*}_{\sim}& \bPerfDA(K,\Lambda)\\	
	}
	$$
	inducing a natural transformation
	$$
	\xi^\flat\circ\R\Perf_*\Rightarrow \R\Perf_*\circ\L\flat^*\circ\xi.
	$$
	We therefore obtain a natural transformation $$\alpha\colon\xi^\flat\circ\R\Perf_*\circ\R j_*\circ\L\iota^*\Rightarrow  \R\Perf_*\circ\L\flat^*\circ\R j_*\circ\L\iota^*\circ\xi.$$ It can be pre-composed with the natural transformation $$\beta\colon\xi^\flat\Rightarrow\xi^\flat\circ \R\Perf_*\circ\R j_*\circ\L\iota^*$$ obtained from $\L j^*\circ\L \Perf^*\Rightarrow\L \iota^*$ (which is induced by the canonical projection to $X$  from the perfection $X^{\Perf}$). We recall that the tilting equivalence is given from left to right by the functor $ \R\Perf_*\circ\L\flat^*\circ\R j_*\circ\L\iota^*$. 
	It suffices then to show that $$\alpha\circ\beta\colon\xi^\flat\Rightarrow\R\Perf_*\circ\L\flat^*\circ\R j_*\circ\L\iota^*\circ\xi$$ is invertible. Using the same strategy as the proof of Theorem \ref{thm:perfectoid_comparison}  we can consider separately the case $\La=\Q$ and the case $\La=\Z/\ell$. The first case is dealt with in \cite{vezz-tilt4rig} while for the second we can invoke the Rigidity Theorem, and the fact that in this case $  \R\Perf_*\circ\L\flat^*\circ\R j_*\circ\L\iota^*$ is just  the tilting of Artin motives $\bD(K_{\et},\La)\cong\bD(K^\flat_{\et},\La)$.%
\end{proof}

\appendix
\section{Rigidity for  rigid motives with transfers over $K$} \label{sec:thm_transfers}

In this appendix, we present an alternative, more ``geometric'' proof of the Rigidity Theorem for $\bRigDM_{\et}(K,\Lambda)$ with $\Lambda=\Z/\ell^N$, where we write as usual $\bRigDM_{\et}(K,\Lambda)$ for $\bRigDM_{\et}(\Sp K,\Lambda)$. This proof  has the advantage of being independent on the full   algebraic Rigidity Theorem in its relative form, but relies only on its version over fields for $\bDM_{\et}(K,\La)$ \cite[Corollary 4.8, Theorem 9.35]{mvw} and for $\bDA_{\et}(K,\La)$ \cite[Lemme 4.6]{ayoub-etale}.

We remark that the following proof is just an adaptation of a proof by Ayoub \cite[Théorème 2.5.34]{ayoub-rig} mixed with the results on $\ell'$-alterations of Gabber and Temkin. We will  heavily refer to \cite[Théorème 2.5.34]{ayoub-rig} and only explain the points where the argument needs to be adapted to our situation. 

As usual, we let $K$ be a complete non-Archimedean valued field, we suppose that $\ell$ is a prime which is invertible in the residue field of $K$ and that $K$ has a finite $\ell$-cohomological dimension. From now on, we let $\Lambda$ be $\Z/\ell^N$ for some $N\in\N_{>0}$.

\begin{thm} \label{thm:rigidity_DM} 
	The functor $\L\iota^* \colon\bD(K_{\et},\Lambda)\ra\bRigDM_{\et}(K,\Lambda)$ is an equivalence of categories.
\end{thm}

\begin{proof} Rigidity for $\bDM_{\et}(K,\La)$ implies Proposition \ref{prop:iso_muDM} in the case $S=\Spa K$ using the derived analytification functor $\L\An^*$. We then know that $\L\iota^*$ is fully faithful (\cf Theorem \ref{thm:embDM}). In order to conclude the theorem, it suffices to show that  a motive of the form $\Lambda^{\tr}(X)$ is Artin, for any given smooth rigid variety $X$ over $K$.

 We remark that motives which are potentially of good reduction (see the definition before \cite[Théorème 2.5.34]{ayoub-rig}) are Artin: this follows  by combining Lemma \ref{lemma:fin_ext} with Lemma \ref{lemma:grisart}. It then suffices to show that $\Lambda^{\tr}(X)$ is in the triangulated category with small sums generated by motives which are potentially of good reduction.

To this aim, it suffices to follow verbatim the proof of  \cite[Théorème 2.5.34]{ayoub-rig} with the following slight changes.

- By Lemmas \ref{lemma:p_extension} and \ref{lemma:fin_ext} a motive is Artin if and only if its base change over a finite extension $K'/K$ is. This replaces the first step of \cite[Théorème 2.5.34]{ayoub-rig}.

- The category $\bRigDM_{\et}(K,\La)$ is a Verdier localization of $\bRigDM_{\mathrm{Nis}}(K,\La)$. In particular, Nisnevich weak equivalences are \'etale weak equivalences and Nisnevich squares give rise to exact triangles also in $\bRigDM_{\et}(K,\La)$ (see \cite[Th\'eor\`eme 2.5.12]{ayoub-rig}).

- Artin motives are closed under tensor products, as $\L\iota^*$ is a monoidal functor. 

- In the third step of the proof of \lc one can require that the alteration $e\colon Y\to X$ (following the notation in \lc) has degree $d$ which is coprime to $\ell$ on the dense open sets  where it is finite, using  \cite[Theorem 5.2.18]{temkinetal} in place of the alteration proved in \cite{berk-sm} and using Lemma \ref{lemma:fin_ext}. We remark that if $e'\colon Y\to X$ is a finite morphism of analytic varieties of degree $d$ coprime to $\ell$, then the motive $\Lambda^{\tr}(X)$ is a direct summand of $\Lambda^{\tr}(Y)$ since $e'\circ e'^{\tr}=d\cdot\id$  and $d\in\La^*$, where $e'^{\tr}$ is the transpose of $e'$ lying in $\Cor(X,Y)$. In particular, $\Lambda^{\tr}(X)$ is Artin if $\Lambda^{\tr}(Y)$ is. {Then, the fact that $\Lambda^{\tr}(Y)$ is an Artin motive when $Y$ has poly-stable reduction can be proved in the same way as it is proved in the third step of the proof of \cite[Théorème 2.5.34]{ayoub-rig}.}
\end{proof}

The following lemmas were used in the previous proof.

\begin{lemma}\label{lemma:grisart}
Let $X$ be the generic fiber of a smooth formal scheme over $\cO_K$. The motive $\Lambda^{\tr}(X) \in \bRigDM_{\et}(K,\Lambda)$   is an Artin motive.
\end{lemma}
\begin{proof}We follow the notation introduced before Corollary \ref{cor:tiltandxi}. 
By hypothesis the motive $\Lambda(X)\in\bRigDA_{\et}(K,\La)$ lies in the image of the functor $ \mathbb{L}(-)^*_\eta\colon \bFormDA_{\et}(\cO_K,\La)\to\bRigDA_{\et}(K,\La)$. We remark that the category $\bFormDA_{\et}(\cO_K,\La)\cong\bDA_{\et}(k,\La)$ is generated by  Artin motives as a consequence of \cite[Lemme 4.6]{ayoub-etale}. Therefore, by the commutativity of the diagram
\[
\begin{tikzpicture}
\matrix(m)[matrix of math nodes,
row sep=1.5em, column sep=2.8em,
text height=1.5ex, text depth=0.25ex]
{ \bD(k_{\et}, \La) & \bDA_{\et}(k, \La) \\
  \bD(\cO_{K\et}, \La) & \bFormDA_{\et}(\cO_K, \La) \\
  \bRigDA^\art(K,\La) & \bRigDA_{\et}(K, \La) \\};
\path[->,font=\scriptsize]
(m-1-1) edge node[auto] {$\L{\iota}^*$} (m-1-2);
\path[->,font=\scriptsize]
(m-1-1) edge node[auto] {$\xi$} (m-2-1);
\path[->,font=\scriptsize]
(m-1-2) edge node[auto] {$\xi$} (m-2-2);
\path[->,font=\scriptsize]
(m-2-1) edge node[auto] {$\L\iota^*$} (m-2-2);
\path[->,font=\scriptsize]
(m-2-1) edge node[auto] {$\L (-)_\eta$} (m-3-1);
\path[->,font=\scriptsize]
(m-2-2) edge node[auto] {$\L (-)_\eta$} (m-3-2);
\path[->,font=\scriptsize]
(m-3-1) edge node[auto] {$\L\iota^*$} (m-3-2);
\end{tikzpicture}
\]
it  follows that $\Lambda(X)$ lies in the category generated by the image of $\L\iota^*\circ\L(-)_\eta\circ\xi$ which is contained in Artin motives. This shows in particular that $\Lambda^{\tr}(X)$ is an Artin motive.
\end{proof}

\begin{lemma}\label{lemma:p_extension}%
	Let $L/K$ be a finite purely inseparable extension of fields, then there is a commutative diagram of functors 
	\[
	\begin{tikzpicture}
	\matrix(m)[matrix of math nodes,
	row sep=1.5em, column sep=2.8em,
	text height=1.5ex, text depth=0.25ex]
	{ 
		\bD(K_{\et},\La) & \bRigDM_{\et}(K, \Lambda) \\
		\bD(L_{\et},\La) & \bRigDM_{\et}(L, \Lambda) \\};
	\path[->,font=\scriptsize]
	(m-1-1) edge node[auto] {} (m-1-2);
	\path[->,font=\scriptsize]
	(m-1-1) edge node[auto] {} (m-2-1);
	\path[->,font=\scriptsize]
	(m-1-2) edge node[auto] {} (m-2-2);
	\path[->,font=\scriptsize]
	(m-2-1) edge node[auto] {} (m-2-2);
	\end{tikzpicture}
	\]
	where the vertical maps are equivalences and the horizontal are  fully faithful functors.
\end{lemma}
\begin{proof}
	The fact that the horizontal functors are fully faithful can be shown by means of Proposition \ref{prop:iso_muDM} (see the proof of Theorem \ref{thm:embDM}). 
By Proposition 2.2.22 of \cite{ayoub-rig} there is an equivalence of categories 
\[ \Cor(K, \La) \stackrel{\sim}{\to} \Cor(L, \La). \]
This equivalence induces the  equivalence $\bRigDM_{\et}(K, \Lambda) \to \bRigDM_{\et}(L, \Lambda)$ of the claim, which obviously preserves Artin motives.
\end{proof}

\begin{lemma} \label{lemma:fin_ext}
	Let $S$ be a smooth rigid variety over $K$, $L/K$ a finite separable extension and $S_L$ the base change of $S$ to $L$. Then $\La^\Tr_L(S_L)$ is an Artin motive if and only if $\La^\Tr_K(S) $ is an Artin motive.
\end{lemma}
\begin{proof}
	We can consider the adjoint pair $$e_\sharp\colon \bRigDM_{\et}(L,\Lambda)\rightleftarrows\bRigDM_{\et}(K,\Lambda)\colon e^*$$
	induced by the smooth map $\Spec L\ra\Spec K$  and observe that it restricts to Artin motives. In particular, if $\La^\Tr(S)$ is Artin then also $\La^\Tr(S_L)$ is. Vice-versa, we can assume that $L/K$ is a normal extension and we suppose that  $\La^\Tr(S_L)$ is Artin over $L$. Then, also $e_\#(\La^\Tr(S_L))=\Lambda_K^{\tr}(S_L)$ is Artin. We can then consider the \u{C}ech hypercover $\cU^\bullet$ of $S$ induced by $S_L\ra S$ and remark that at each stage it is isomorphic to a disjoint union of copies of $S_L$. In particular the simplicial motive $\Lambda^{\tr}(\cU^\bullet)$ is lewel-wise Artin, and hence also $\Lambda(S)\cong\hocolim \Lambda^{\tr}(\cU^\bullet)$ is, since Artin motives are closed under sums and cones.
\end{proof}

%\vspace{cm}

%	

\end{document}